\documentclass[aihp]{imsart}

\RequirePackage{amsthm,amsmath,amsfonts,amssymb}
\RequirePackage[numbers]{natbib}
\RequirePackage[colorlinks,citecolor=blue,urlcolor=blue]{hyperref}
\RequirePackage{graphicx}

\RequirePackage{bm}
\RequirePackage{enumerate}
\RequirePackage{bbm}
\RequirePackage{comment}

\startlocaldefs
\theoremstyle{plain}

\newtheorem{theorem}{Theorem}[section]
\newtheorem{lemma}[theorem]{Lemma}

\newtheorem{corollary}[theorem]{Corollary}
\newtheorem{proposition}[theorem]{Proposition}
\newtheorem{remark}[theorem]{Remark}

\theoremstyle{remark}

\newcommand{\ndN}{\mathbb{N}}
\newcommand{\ndZ}{\mathbb{Z}}

\newcommand{\ndR}{\mathbb{R}}
\newcommand{\ndC}{\mathbb{C}}

\renewcommand{\Pr}[1]{\mathbb{P}(#1)}

\newcommand{\Prb}[1]{\mathbb{P}\left(#1\right)}

\newcommand{\Ex}[1]{\mathbb{E}[#1]}

\newcommand{\Exb}[1]{\mathbb{E}\left[#1\right]}

\newcommand{\Va}[1]{\mathbb{V}[#1]}

\newcommand{\one}{{\mathbbm{1}}}

\newcommand{\convdis}{\,{\buildrel d \over \longrightarrow}\,}

\newcommand{\convp}{\,{\buildrel p \over \longrightarrow}\,}

\newcommand{\eqdist}{\,{\buildrel d \over =}\,}

\newcommand{\atv}{\,{\buildrel d \over \approx}\,}

\newcommand{\cE}{\mathcal{E}}
\newcommand{\cF}{\mathcal{F}}
\newcommand{\cG}{\mathcal{G}}

\endlocaldefs

\begin{document}

\begin{frontmatter}
\title{Gibbs partitions: a comprehensive phase diagram}
\runtitle{Gibbs partitions: a comprehensive phase diagram}

\begin{aug}
\author[A]{\inits{B.}\fnms{Benedikt}~\snm{Stufler}\ead[label=e1]{benedikt.stufler@tuwien.ac.at}}
\address[A]{
Technische Universit\"at Wien, Austria\printead[presep={,\ }]{e1}}
\end{aug}

\begin{abstract}
	We study  Gibbs partition models, also known as composition schemes.  Our main results comprehensively describe their phase diagram, including a phase transition from the convergent case described in Stufler (2018, Random Structures \& Algorithms) to a new dense regime characterized by a linear number of components with fluctuations of smaller order quantified by  an $\alpha$-stable law for $1< \alpha \le 2$. We prove a functional scaling limit for a process whose jumps correspond to the component sizes and discuss applications to extremal component sizes.  At the transition we observe a mixture of the two asymptotic shapes. We also treat extended composition schemes and  prove a local limit theorem in a dilute regime with the limiting law being related to an $\alpha$-stable law for $0< \alpha < 1$. We describe the asymptotic size of the largest components via a point process limit.
\end{abstract}




\end{frontmatter}

\section{Introduction}

Combinatorial structures of all kinds may often be decomposed into smaller components. When studying random discrete structures it is hence of interest to understand the behaviour of the number and sizes of these components. The large quantity and variety of classes of graphs, trees, permutations, etc.  motivates the development of a general theory that encompasses these models as special cases. 

A successful approach in this regard is the series of works~\cite{zbMATH01905956,MR2032426,MR2121024} which study random partitions satisfying a conditioning relation.  Another important approach is that of the Gibbs partition model~\cite{MR2245368,zbMATH05166312,MR2453776}. Both models appear to be quite general. They have a non-trivial intersection, but neither  encompasses the other. In combinatorial literature~\cite{MR1871555,MR2483235}, the term composition schema is often used  to refer to a Gibbs partition, since the generating series of its partition function is the functional composition of the generating series of the two associated weight sequences.

Works studying composition schemes include~\cite{zbMATH01135939,  MR1871555,zbMATH05166312, MR2453776, zbMATH06964621, zbMATH07359155, zbMATH07235577, 2021arXiv210303751B}. The present work provides a comprehensive phase diagram for their asymptotic shape, including new results and phases. Forthcoming independent works~\cite{mmca,mmcb} pursue an alternative analytic approach to this topic and provide combinatorial applications. Let us briefly outline our main contributions and provide context by detailing pre-existing results:

\paragraph*{The dense case}

We define and study a dense case, in which the number $N_n$ of components of the Gibbs partition model concentrates at $n / \mu$ for a specific constant $\mu>0$. We prove that the fluctuations have order $n^{1/\alpha}$ for some $1< \alpha \le 2$.  After a proper rescaling, they approach an $\alpha$-stable law. In fact, in Theorem~\ref{te:dense} we establish a precise local limit theorem of the form
\[
\lim_{n \to \infty} \sup_{\ell \ge \delta n} \left| L(n) n^{1/\alpha} \Pr{ N_n = \ell} - h\left( \frac{\ell - n/\mu}{L(n) n^{1/\alpha}}  \right)  \right| =0,
\]
with $h$ denoting the density of an  $\alpha$-stable law  and $L(n)$ denoting a slowly varying factor associated to the weight-sequences. An important ingredient in its proof are new recent results by~\cite{zbMATH07179510} on local probabilities for randomly stopped sums. We also study the sizes $K_1, \ldots, K_{N_n}$ of the components and establish in Theorem~\ref{te:functional} a functional limit theorem
\begin{align*}
	\left( \frac{\sum_{i=1}^{\lfloor s N_n \rfloor}K_i - N_n s \mu }{L(n)n^{1/\alpha}}, \,\, 0 \le s \le 1 \right) \convdis (\mu Y_s, \,\, 0 \le s \le 1)
\end{align*}
in the space $\mathbb{D}([0,1], \ndR)$ of càdlàg functions as $n \to \infty$. Here $(Y_s)_{s \ge 0}$ denotes the spectrally positive L\'evy process with Laplace exponent $\Exb{e^{-t Y_s }} = \exp(s t^\alpha)$.  The functional limit theorem entails, for example, that the  largest component  admits the largest jump of $(Y_s)_{0 \le s \le 1}$ as scaling limit. Hence
\begin{align*}
	\frac{1}{L(n)n^{1/\alpha}} \max_{1 \le i \le N_n} K_i \convdis \begin{cases}W_1, &1<\alpha<2 \\ 0, &\alpha=2 \end{cases}
\end{align*}
for a random variable $W_1 \ge 0$ with cumulative distribution function  $\exp\left(\frac{\mu^\alpha}{\Gamma(1 - \alpha)} x^{-\alpha} \right)$. In case $1<\alpha<2$ we additionally get scaling limits for the $j$th largest component for each integer $j \ge 1$, see Corollary~\ref{co:jump}. The reason why the limit is degenerate in the case $\alpha=2$ is because in this case the process $(Y_s)_{s \ge 0}$ is continuous. We can provide more precise information in this case. In fact, our main observation Lemma~\ref{le:independent} for the proof of the functional limit theorem states that for any sequence $(s_n)_{n \ge 1}$ of real numbers satisfying
$\frac{s_n}{L(n)n^{1/\alpha}} \to \infty$ and  $s_n  = o(n)$
we have
\begin{align*}
	\lim_{n \to \infty}	d_{\mathrm{TV}}\left( (K_1, \ldots, K_{\min(N_n, \lfloor n / \mu - s_n\rfloor)}), (X_1, \ldots, X_{\lfloor n/ \mu - s_n\rfloor} ) \right)  = 0
\end{align*}
for an independent identically distributed family $(X_i)_{i \ge 1}$ of random variables. This has far-reaching applications: any functional that typically does not get perturbed by $o(n)$ coordinates behaves like a functional of $n/\mu$ i.i.d. random variables. This allows us to transfer a multitude of asymptotics for i.i.d. variates to the Gibbs partition model with little effort. As an example, we provide applications to order statistics, see Corollaries~\ref{co:maxcomp} and~\ref{co:profile}.

We note that in the dense case, the composition scheme associated to the Gibbs partition may be critical or supercritical. 

\paragraph*{The convergent case}

The convergent case exhibits a strong gelation phenomenon, as a unique giant component emerges. The number of components admits a finite distributional limit
\[
N_n \convdis \hat{N},
\]
and  the collection of non-maximal components also admits a finite limit with  Boltzmann type distribution. Specifically, in Theorem~\ref{te:convergent} we show that the
result $\phi(K_1, \ldots, K_{N_n})$ of replacing the left-most maximal component by a placeholder value $*$ satisfies
\begin{align*}
	\phi(K_1, \ldots, K_{N_n}) \convdis \left(X_1, \ldots, X_{J-1}, *, X_{J}, \ldots, X_{\hat{N}-1}\right),
\end{align*}
for an independent identically distributed family $(X_i)_{i \ge 1}$ of finite random variables and a conditionally independent uniformly selected index $J \in \{1, \ldots, \hat{N}-1\}$. Thus, there is an asymptotically unique giant component whose size $M_n$ satisfies
\begin{align*}
	n - M_n \convdis \sum_{i=1}^{\hat{N}-1} X_i.
\end{align*}
Such a behaviour was observed in~\cite{MR2121024} for the mentioned models of random partitions satisfying a conditioning relation. To this end, the work~\cite{MR2121024} uses   a  perturbed Stein recursion approach. For Gibbs partitions with a subcritical composition scheme such a behaviour was later studied in~\cite{zbMATH06964621}, using the theory of subexponential probability distributions~\cite{MR0348393, MR714482,MR772907,MR3097424}.  The reason for revisiting this setting is that since the writing and publication of~\cite{zbMATH06964621}, new results~\cite{zbMATH07179510} on local probabilities for randomly stopped sums have been established, allowing us to additionally treat Gibbs partitions with a critical composition scheme.

Prior work in the analytic combinatorics literature treat cases of subcritical composition schemes, assuming that the generating series of the ``inner'' weight-sequence admits a suitable singularity expansion. Specifically, in~\cite[Thm. 1]{zbMATH01135939}  the size of the largest component was determined, and \cite[Prop. IX.1]{MR2483235} determined the asymptotic number of components under similar analytic assumptions. The approach is based on singularity analysis and analytic methods.

The convergent case studied in the present work via probabilistic methods encompasses these settings  and additionally shows that there is a  limit distribution for the small fragments. Furthermore, we observe that the convergent case includes both subcritical and critical composition schemes. This is important from a conceptual viewpoint, since   critical composition schemes  exhibiting gelation appear to be a new addition to the analytic combinatorics literature.

\paragraph*{The mixture case}

In the mixture case, we establish in Theorem~\ref{te:mixture} that the Gibbs partition behaves  with a limiting probability $p \in ]0,1[$ as in the dense case, and with limiting probability $1-p$ as in the convergent case. Here \[
p = \lim_{n \to \infty} \Pr{N_n \ge n / (2 \mu^{-1})},
\] and all mentioned limit theorems for the number and sizes of components in the dense case hold for the conditioned Gibbs partition \[
((K_1, \ldots, K_{N_n}) \mid N_n \ge n / (2 \mu^{-1})).
\] 
In particular, there is an associated parameter $\alpha \in ]1, 2]$, and a local limit theorem for $N_n$ with an $\alpha$-stable limit.
Likewise, all mentioned limit theorems for the convergent case hold for the conditioned partition
\[
((K_1, \ldots, K_{N_n}) \mid N_n < n / (2 \mu^{-1})).
\]

The case $\alpha = 3/2$ has been previously studied using singularity analysis and saddle-point bounds in the well-known work~\cite{MR1871555}, alongside important applications to various combinatorial models. Specifically,~\cite[Thm. 5]{MR1871555} establishes the asymptotic local probabilities for $\Pr{N_n = \ell}$, assuming that the generating series of the weight-sequences satisfy suitable singularity expansions that result in asymptotic power laws for the coefficients.

In the present work we treat a more general setting, allowing regularly varying weight-sequences and domains of attractions of $\alpha$-stable laws for $1 < \alpha \le 2$. We additionally obtain precise information on the component sizes, showing that a giant component arises in the case $N_n < n / (2 \mu^{-1})$, with a finite distributional limit for the collection of non-maximal components. The functional limit theorem for the component sizes in the case $N_n \ge n / (2 \mu^{-1})$ is also new, even in the case $\alpha = 3/2$, as are the precise limit distribution for the second largest component and the total variational approximation of a linear fraction of components.

\paragraph*{The dilute case}

In the dilute case the number of components spreads out at the scale $n^{1/\alpha}$ for $0 < \alpha < 1$,  without any centring. In Theorem~\ref{te:dilute} we prove for each $\delta>0$ the local limit theorem
\begin{align*}
	\lim_{n \to \infty} \sup_{\ell \ge \delta n^{\alpha}} \left|n^{\alpha} \Pr{ N_n = \ell} - \tilde{f}(\ell / n^{\alpha})  \right| =0
\end{align*}
for a density function
\[
\tilde{f}(x) = \frac{1}{\alpha \Ex{(X_{\alpha}(\gamma, 1,0))^{\alpha(b-1)}}} \frac{f\left(\frac{1}{x^{1/\alpha}}\right)}{x^{b+1/\alpha}}
\]
constructed from the density $f$ of an $\alpha$-stable random variable $X_{\alpha}(\gamma, 1,0)$ with a scale parameter $\gamma>0$. Thus
\[
	N_n / n^{\alpha} \convdis Z
\]
for a random variable $Z>0$ with density function $\tilde{f}$. In Corollary~\ref{co:poiz} we deduce that if $k_1, k_2, \ldots$ is a sequence of positive integers with $k_n  \ll n^{\frac{\alpha}{1+\alpha}}$, then
\[
	\frac{\#_{k_n} P_n}{\Pr{X=k_n}n^{\alpha}} \convdis Z.
\]
Whereas, for $k_n \sim \left( \frac{c_w}{\upsilon W(\rho_w)} \right)^{\frac{1}{1+\alpha}} n^{\frac{\alpha}{1+\alpha}}$ with $0< \upsilon < \infty$ (and $c_w$, $W(\rho_w)$ denoting constants associated to the Gibbs partition model)
\[
		\#_{k_n} P_n \convdis \mathrm{Poi}(\upsilon Z).
\]
For $k_n \gg n^{\frac{\alpha}{1+\alpha}}$ we obtain
\[
		\#_{k_n} P_n \convdis 0.
\]
A similar local limit theorem for the number of components and similar component size asymptotics were recently established in~\cite[Thm. 4.1, Thm. 5.1]{2021arXiv210303751B}  alongside  important combinatorial applications. We also treat extended composition schemes considered in~\cite{2021arXiv210303751B}, see Section~\ref{sec:product}. 

Although $n^{\frac{\alpha}{1+\alpha}}$ appears to be a threshold for $\#_{k_n} P_n$, we show that maximal component sizes actually scale at the order $n$. In Corollary~\ref{co:pointprocess} we show that 
\[
	\sum_{\substack{1 \le i \le N_n \\ K_i > 0}} \delta_{K_i / n} \convdis \Upsilon
\]
for a point process $\Upsilon$ on $]0,1]$ with intensity 
\[
\frac{x^{-\alpha-1} (1-x)^{\alpha(2-b)-1}}{B(1-\alpha, \alpha(2-b))  }\,\mathrm{d}x.
\]
Here $B(\cdot, \cdot)$ denotes Euler's beta function. Consequently, the maximal component sizes $K_{(1)} \ge K_{(2)} \ge \ldots$ of $P_n$ converge jointly to the ranked points $\eta_1 \ge \eta_2 \ge \ldots$ of $\Upsilon$ after rescaling by $n^{-1}$, that is,
\[
	K_{(k)}/n \convdis \eta_k, \qquad k \ge 1 (\text{ jointly}).
\]
We verify that $\eta_k>0$ almost surely for each $k\ge1$, and in  Proposition~\ref{pro:explicit} we determine that the distribution of $\eta_k$ is given by
	\begin{multline*}
		\Pr{\eta_k < x} =  \frac{ \Gamma(1- \alpha(b-1)) }{ 2\pi |\Gamma(-\alpha)|^{b-1} \Gamma(2-b)}  \\\int_0^\infty \int_{-\infty}^\infty   u^{\alpha(b-1)} \exp\left( -\frac{1}{u^\alpha}\int_I \frac{e^{ityu}}{y^{1+\alpha}}\,\mathrm{d}y -iut -  |\Gamma(-\alpha)| (-it)^\alpha\right)   \sum_{j=0}^{k-1}\frac{1}{j!} \left( \frac{1}{u^\alpha}\int_I \frac{e^{ityu}}{y^{1+\alpha}} \,\mathrm{d}y \right)^j  \,\mathrm{d}t\,\mathrm{d}u
	\end{multline*}
for $0 < x \le 1$. 

\paragraph*{The expansive case}

The work~\cite{MR2453776} studies a Gibbs partition model corresponding to a subcritical composition scheme $\exp(W(z))$ with expansive coefficients $[z^n]W(z) \sim C n^{-a}$ for an arbitrary exponent $-\infty < a<1$. The expansive case of~\cite{MR2453776} constitutes a separate phase, that is not encompassed by the results of the present work.

\paragraph*{The superexponential case}

The series of works~\cite{MR0205883,MR0211880,MR0229546} studies composition schemes of the form $\exp(W(z))$ and determines conditions under which the associated Gibbs partition consists of a single component with probability tending to $1$ as the size tends to infinity.  In \cite[Sec. 6.5.2]{zbMATH07235577}  extensions of this result are provided, in particular to  composition schemes of the form $V(W(z))$ with $V(z)$ having positive radius of convergence and $W(z)$ having radius of convergence $0$. Weight sequences with non-analytic generating functions have also appeared in the probabilistic literature on random trees. In particular,~\cite{MR2860856,MR2908619} study recursive structures with superexponential weights, which give rise to composition schemes where both $V(z)$ and $W(z)$ have radius of convergence zero.

\paragraph*{Extended composition schemes}

All Gibbs partitions studied in this work (dense, convergent, mixture, and dilute cases) have the property, that their partition function satisfies a subexponentiality condition. This entails that if we study so-called extended composition schemes
\[
H(z)V(W(z)),
\]
then the limiting shape is a (possibly degenerate) mixture of a Boltzmann distribution and the limiting shape of the Gibbs partition corresponding to the schema $V(W(z))$, regardless whether it belongs to the dense, convergent, mixture, or dilute case. See Lemma~\ref{te:extended} and Corollary~\ref{co:extended} for details. Extended composition schemes related to the dilute case have also been studied in~\cite{2021arXiv210303751B} via analytic methods.

\subsection*{Further research directions}

It would be interesting to investigate whether  limiting behaviours related to $\alpha=1$ stable distributions may be observed under certain conditions. These kind of distributions have received recent attention in the context of random discrete structures~\cite{zbMATH07057468}. Looking at the very specific  ``blind spots'' that do not satisfy the assumptions in the theorems presented here gives some clues where such a transition \emph{might} occur, but we refrain from making concrete predictions. We also note that it would   be interesting to pursue the study of the expansive case treated in~\cite{MR2453776} further and extend it to more settings. Furthermore, there is a wealth of combinatorial models that fit into the Gibbs partition setting, see for example~\cite{MR2245368,MR2483235,MR1871555,2021arXiv210303751B}. It appears that the results of the present work may aid in proving new and interesting properties for some of these. We hope to pursue this in future work.

\subsection*{Plan of the paper}
Section~\ref{sec:prel} recalls the definition of the Gibbs partition model and fixes some notation. Section~\ref{sec:mainresults} states all main results. Specifically, Subsection~\ref{sec:dense} treats the dense case, Subsection~\ref{sec:convergent} treats the convergent case, Subsection~\ref{sec:mixture} treats the mixture case, Subsection~\ref{sec:dilute} treats the dilute case, and Subsection~\ref{sec:product} treats product structures and extended composition schemes. In Section~\ref{sec:proofs} we provide proofs of all results.

\section{Preliminaries}

\label{sec:prel}

\subsection{The Gibbs partition model}
\label{sec:gibbs}

Let $S$ be a finite non-empty set. A \emph{partition} of $S$ is a collection $P = \{S_1, \ldots, S_k\}$ of subsets of $S$ that are pairwise disjoint and satisfy $S = \bigcup_{i=1}^k S_i$. The elements of $P$ are its \emph{components}, and $k \ge 1$ is hence its number of components. We may form the collection $\mathrm{Part}(S)$ of all partitions of $S$.

Suppose that we are given sequences $\bm{v} = (v_i)_{i \ge 1}$ and $\bm{w} = (w_i)_{i \ge 0}$ of non-negative real numbers. This allows us to define the \emph{weight}
\[
u(P) =  |P|!v_{|P|} \prod_{Q \in P} |Q|!w_{|Q|} 
\]
of the partition $P$. Let $n \ge 1$ denote the number of elements of $S$. We define the  \emph{partition function}   by
\[
u_n =  \frac{1}{n!}\sum_{P \in \mathrm{Part}(S)} u(P).
\]
Note that $u_n \ge 0$  only depends on the cardinality $n$ of $S$.
The relation between the \emph{generating series} $V(x) = \sum_{i \ge 1} v_i x^i$, $W(x) = \sum_{i \ge 0} w_i x^i$, and $U(x) = \sum_{i \ge 0} u_i x^i$ is described by the composition
\[
U(x) = V(W(x))
\]
of formal power series. In the combinatorial literature, this relation is also called a \emph{composition schema}. We let $\rho_u, \rho_v, \rho_w \ge 0$ denote the radii of convergence of the series $U(x)$, $V(x)$, and $W(x)$. If these radii are positive, the composition scheme is called \emph{subcritical} if $W(\rho_w) < \rho_v$, \emph{critical} if $W(\rho_w) = \rho_v$, and \emph{supercritical} if $W(\rho_w)>\rho_v$. 

For any integer $n \ge 1$ with $u_n>0$ we may form the associated \emph{Gibbs partition} model, which describes a random element $P_n$ of   $\mathrm{Part}([n])$ for $[n] := \{1, \ldots, n\}$, with 
\[
\Pr{P_n = P} = \frac{u(P)}{n!u_n}
\]
for each $P \in \mathrm{Part}([n])$. We let $N_n$ denote the number of components of $P_n$, and $(K_1, \ldots, K_{N_n})$ their sizes.  Here we order the components in some canonical way, for example according to their smallest elements in increasing order. The tuple $(K_1, \ldots, K_{N_n})$ is exchangeable, hence the choice of ordering is of no particular importance. We let $M_n = \max_{1 \le i \le N_n} K_i$ denote the maximal component size.

If $\#_i P_n$ denotes the number of components of size $i \ge 0$, then for all families $(n_i)_{i \ge 0}$ of non-negative integers with $\sum_{i=0}^n n_i = n$ we have
\[
\Prb{ (\#_i P_n)_{i \ge 0} = (n_i)_{i \ge 0} } = \frac{r! v_{r}}{u_n} \prod_{i \ge 0} \frac{w_i^{n_i}}{n_i!}
\]
with $r := \sum_{i \ge 0} n_i$.

The Gibbs partition model is invariant under \emph{tilting} of the weight sequence $\bm{w}$. That is, for any constant $t > 0$ we may form the tilted sequence $\tilde{\bm{w}} = (\tilde{w}_i)_{i \ge 0}$ with $\tilde{w}_i = w_i t^i$, and the associated weights $\tilde{u}$ satisfy
\[
\tilde{u}(P) = t^n  u(P)
\]
for all $P \in \mathrm{Part}([n])$. Consequently, the associated Gibbs partition model $\tilde{P}_n$ satisfies
\[
P_n \eqdist \tilde{P}_n.
\]

\subsection{Notation}

For $0<\alpha \le 2$, $-1 \le \beta \le 1$, $\gamma>0$, and $-\infty < \delta < \infty$ we let $S_\alpha(\gamma, \beta, \delta)$ denote the $\alpha$-stable  distribution with scale parameter  $\gamma$, skewness parameter $\beta$, and location parameter $\delta$, so that the characteristic function of a $S_\alpha(\gamma, \beta, \delta)$-distributed random variable $X_\alpha(\gamma, \beta, \delta)$ is given by
\[
\Exb{e^{ i t X_\alpha(\gamma, \beta, \delta)}} = \begin{cases}
	\exp\left( - \gamma^\alpha |t|^\alpha \left(1 - i \beta  \mathrm{sgn}(t) \tan\left(\frac{\pi \alpha}{2}\right)\right) + i \delta t \right), &\alpha \ne 1 \\
	\exp\left(- \gamma |t| \left( 1 + i \beta \mathrm{sgn}(t) \frac{2}{\pi}  \log( |t|)\right) + i \delta t\right), &\alpha = 1.
\end{cases}
\]
Recall that a function $L: \ndR_{>0} \to \ndR_{>0}$ is  \emph{slowly varying}, if for all $t>0$
\[
\lim_{x \to \infty} \frac{L(tx)}{L(x)} = 0.
\]
For any $a\in \ndR$ the product  $L(x)x^a$   of a slowly varying function $L$ with a power of~$x$ is called \emph{regularly varying} with index~$a$.

We let $\ndN = \{1, 2, \ldots\}$ denote the set of positive integers, and $\ndN_0 = \ndN \cup \{0\}$ the set of non-negative integers. The $n$-th coefficient of a power series $f(z)$ is denoted by $[z^n]f(z)$.
Throughout the following all unspecified limits are taken as $n \to \infty$. For sequences $(a_n)_{n \ge 1}$, $(b_n)_{n \ge 1}$ of real numbers we write
\begin{align*}
	a_n \sim b_n & \qquad\text{if}\qquad \lim_{n \to \infty} a_n / b_n = 1, \\
	a_n = o(b_n) & \qquad\text{if}\qquad  \lim_{n \to \infty} a_n / b_n = 0, \\
	a_n \gg b_n & \qquad\text{if}\qquad  \lim_{n \to \infty} a_n / b_n = \infty, \\
	a_n = O(b_n) & \qquad\text{if}\qquad  \limsup_{n \to \infty} |a_n| / |b_n| < \infty \\
	a_n = \Theta(b_n) & \qquad\text{if}\qquad  a_n = O(b_n) \text{ and } b_n = O(a_n).
\end{align*}
We let $\convdis$ denote convergence in distribution, also called weak convergence of random variables. Convergence in probability is denoted by $\convp$. The total variation distance between (the laws of) two random variables $X$ and $Y$ is denoted by $d_{\mathrm{TV}}(X,Y)$. We write $X_n \atv Y_n$ if $d_{\mathrm{TV}}(X_n, Y_n) \to 0$ as $n \to \infty$. We let $O_p(1)$ denote a stochastically bounded random variable.

\section{Main results}
\label{sec:mainresults}

In the following sections we determine a comprehensive phase diagram for the asymptotic shape of the Gibbs partition~$P_n$ as $n \to \infty$.  

\subsection{The dense case}

\label{sec:dense}

In dense case the number of components in the Gibbs partition $P_n$ concentrates at a constant multiple of $n$, with fluctuations of order $n^{1/\alpha}$ quantified asymptotically by a stable law with index $1< \alpha \le 2$.

\begin{theorem}
	\label{te:dense}
	Suppose that $\rho_v>0$ and
	\[
	v_n = L_v(n) n^{-b}\rho_v^{-n} 
	\]
	for a slowly varying function $L_v$ and an exponent $b>1$. Furthermore suppose that one the following  cases holds.
	\begin{enumerate}[\qquad i)]
		\item 	We have $\rho_v = W(\rho_w)$ and 
		\[
		w_n = L_w(n) n^{- a} \rho_w^{-n}
		\]	
		for a slowly varying function  $L_w$ and an exponent $a>2$, such that  $1 < b <a$, or  $b=a$ and  $L_w(n) = o(L_v(n))$.
		\item We have $\rho_v < W(\rho_w)$ and $\gcd\{ n \ge 0 \mid w_n > 0\} = 1$.
	\end{enumerate}
	Set $\alpha = \min(a-1,2)$ in the first case, and $\alpha = 2$ in the second. Let 
	\[
	\mu = W(\rho_u)^{-1} \sum_{n \ge 1 } n w_n \rho_u^n.
	\] Then there exist a constant $\gamma > 0$ and a slowly varying function~$L$ specified in Equations~\eqref{eq:defofl1} and~\eqref{eq:defofgamma} below with
	\begin{align}
		\label{eq:denseclt}
		\frac{N_n - n/\mu}{L(n) n^{1/\alpha}} \convdis X_\alpha(\gamma, -1, 0)
	\end{align}
	as $n \to \infty$.
	Let $h$ denote the density function of $X_\alpha(\gamma, -1, 0)$. Then for any $\delta >0$
	\begin{align}
		\label{eq:densellt}
		\lim_{n \to \infty} \sup_{\ell \ge \delta n} \left| L(n) n^{1/\alpha} \Pr{ N_n = \ell} - h\left( \frac{\ell - n/\mu}{L(n) n^{1/\alpha}}  \right)  \right| =0.
	\end{align}
\end{theorem}

We may describe the functions and constants of these statements in more detail. The assumptions of Theorem~\ref{te:dense} allow us to define  a random non-negative integer~$X$ with probability generating function
\begin{align}
	\label{eq:defX}
	\Ex{z^X} = \frac{W(\rho_u z)}{W(\rho_u)},
\end{align}
so that $\mu = \Ex{X}$. For all $x \ge 0$ let $K(x) = \Ex{X^2 \one_{X \le x}}$. Moreover, let
\begin{align}
	\label{eq:defofg}
	g(n) = \begin{cases}
		\sqrt{\frac{\Va{X}}{2}}, \quad &\Va{X}<\infty \\
		\sqrt{K\left(\sup\left\{x \ge 0 \mid \frac{K(x)}{x^2} \ge \frac{1}{n} \right\} \right)}, \quad &\Va{X}=\infty \,\,\mathrm{and}\,\, \alpha= 2 \\
		n^{-1/\alpha}(\Gamma(1- \alpha))^{1/\alpha} \inf\left\{ x \ge 0 \mid \Pr{X > x} \le \frac{1}{n} \right\}, &\Va{X}=\infty \,\,\mathrm{and}\,\, 1<\alpha<2.
	\end{cases}
\end{align}
Then we may set
\begin{align}
	\label{eq:defofl1}
	L(n) = \mu^{-1-1/\alpha} g(n)
\end{align}
and
\begin{align}
	\label{eq:defofgamma}
	\gamma = \left( - \cos \frac{\pi \alpha}{2} \right)^{1/\alpha}.
\end{align}
This way, $X_\alpha(\gamma, -1, 0) \eqdist - X_\alpha(\gamma, 1, 0)$, and  $X_\alpha(\gamma, 1, 0)$ has Laplace transform
\begin{align}
	\Exb{e^{-t X_\alpha(\gamma, 1, 0)}}  = \exp(t^\alpha ), \qquad \mathrm{Re}\, t \ge 0.
\end{align}
For $\alpha=2$, the distribution $S_\alpha(\gamma, -1, 0)$ is  Gaussian  with mean $0$ and variance $2$.  Hence
\begin{align}
	h(x) = \frac{1}{2 \sqrt{\pi}} \exp(- x^2 / 4 ).
\end{align}
For $1 < \alpha < 2$, by~\cite[Thm. 4.1]{2011arXiv1112.0220J}
\begin{align}
	h(x) = \frac{1}{\pi x} \sum_{k=1}^\infty \frac{\Gamma(k/\alpha +1)}{k!} (-x)^k \sin\left(- \frac{k \pi}{\alpha} \right).
\end{align}

In case that the slowly varying function $L_w$ admits a limit the associated sequences $g(n)$ and $L(n)$ may be chosen to be constants:
\begin{proposition}
	Suppose that the limit
	\[
	c_w = \lim_{n \to \infty} L_w(n) 
	\]
	exists and is positive. In this case, the slowly varying function $g(n)$ may be determined explicitly in the infinite variance cases:
	\begin{enumerate}[\qquad a)]
		\item  If $1 < \alpha < 2$, then 
		\[
		\Pr{X \ge x} \sim \frac{c_w}{W(\rho_w)\alpha} x^{-\alpha}
		\]
		as $x \to \infty$, and
		\[
		\lim_{n \to \infty} g(n) = \left( \frac{c_w \Gamma(1- \alpha)}{W(\rho_w) \alpha}  \right)^{1/\alpha}.
		\]
		Hence we may set $L(n)$ to the constant
		\[
		L(n) = \mu^{-1-1/\alpha} \left( \frac{c_w \Gamma(1- \alpha)}{W(\rho_w) \alpha}  \right)^{1/\alpha}.
		\]
		\item If $\alpha= 2$ and $\Va{X} = \infty$, then
		\[
		K(x) \sim \frac{c_w}{W(\rho_w)} \log(x)
		\]
		as $x \to \infty$, and
		\[
		g(n) \sim  \frac{1}{2} \sqrt{\frac{c_w n \log n}{W(\rho_w)}}
		\]
		Hence we may define $L(n)$ to equal
		\[
		L(n) =  \frac{1}{2\mu^{1+1/\alpha}} \sqrt{\frac{c_w n \log n}{W(\rho_w)}}.
		\]
	\end{enumerate}
\end{proposition}

Apart from the asymptotic behaviour of the number of components, we may also describe their sizes via a functional limit theorem. To this end we let $\mathbb{D}([0,1], \ndR)$ denote the set of càdlàg functions on the unit interval, equipped with the Skorokhod $J_1$-topology. Let $(Y_s)_{s \ge 0}$ denote the spectrally positive L\'evy process with Laplace exponent
\begin{align}
	\Exb{e^{-t Y_s }} = \exp(s t^\alpha).
\end{align}

\begin{theorem}
	\label{te:functional}
	Under the assumptions of Theorem~\ref{te:dense} we have
	\begin{align}
		\left( \frac{\sum_{i=1}^{\lfloor s N_n \rfloor}K_i - N_n s \mu }{L(n)n^{1/\alpha}}, \,\, 0 \le s \le 1 \right) \convdis (\mu Y_s, \,\, 0 \le s \le 1)
	\end{align}
	in $\mathbb{D}([0,1], \ndR)$ as $n \to \infty$.
\end{theorem}

Such a behaviour has been proven by~\cite[Thm. 3]{MR3335012} for the sizes of the fringe subtrees attached to a vertex with macroscopic degree in non-generic random trees, enabling  combinatorial applications to planar maps~\cite{AddBe} and planar graphs~\cite{zbMATH07235577}. The approach in~\cite[Thm. 3]{MR3335012} makes use of an independence result of a linear fraction of these fringe subtrees, which is  deduced by applying the independence of the small jumps~\cite{MR2775110} in a random walk setting with a unique giant jump. The approach for Theorem~\ref{te:functional} also builds on an independence result for component sizes, see Lemma~\ref{le:independent} below,  but this main lemma is proved differently, using random walks without giant jumps instead.

Theorem~\ref{te:functional} implies, for example, that the rescaled size of the  largest component of the Gibbs partition $P_n$ converges to the size of the largest jump of $(\mu Y_s, \,\, 0 \le s \le 1)$.  Hence we readily obtain:

\begin{corollary}
	\label{co:jump}
	Let $K_{(1)} \ge K_{(2)} \ge \ldots$ denote the component sizes $K_1, \ldots, K_{N_n}$ arranged in decreasing order.
	Under the assumptions of Theorem~\ref{te:dense} we have
	\begin{align}
		\label{eq:jumpsize}
		\frac{1}{L(n)n^{1/\alpha}} K_{(1)} \convdis \begin{cases}W_1, &1<\alpha<2 \\ 0, &\alpha=2 \end{cases}
	\end{align}
	for a Fr\'{e}chet-distributed random variable $W_1 \ge 0$ with cumulative distribution function
	\[
	\Pr{W_1 \le x} = \exp\left(\frac{\mu^\alpha}{\Gamma(1 - \alpha)} x^{-\alpha} \right)
	\]
	for all $x > 0$. Moreover, if $1<\alpha<2$, then for each integer $j \ge 2$
	\[
	\frac{1}{L(n)n^{1/\alpha}} K_{(j)} \convdis W_j
	\]
	for a random variable $W_j$ with density function
	\[
	\alpha \frac{\mu^\alpha}{\Gamma(1 - \alpha)} x^{-\alpha-1} \frac{ \left( \frac{\mu^\alpha}{\Gamma(1 - \alpha)} x^{-\alpha} \right)^{j-1}  }{(j-1)!} \exp\left(- \frac{\mu^\alpha}{\Gamma(1 - \alpha)} x^{-\alpha}\right).
	\]
\end{corollary}

The limit in~\eqref{eq:jumpsize} is degenerate for $\alpha=2$, since the limiting process in Theorem~\ref{te:functional} is continuous in this case. We may nevertheless provide more accurate asymptotics. To this end, let $(X_i)_{i \ge 1}$ denote independent copies of the random variable $X$ described in Equation~\eqref{eq:defX}. 
The following lemma  is used in the proof of Theorem~\ref{te:functional}, but also provides additional information on the component sizes. It ensures that a large part of the components of $P_n$ become independent from each other.

\begin{lemma}
	\label{le:independent}
	Under the assumptions of Theorem~\ref{te:dense} there exists a sequence $(s_n)_{n \ge 1}$ of positive integers with 
	$s_n  = o(n)$
	(and necessarily $\frac{s_n}{L(n)n^{1/\alpha}} \to \infty$) such that
	\begin{align}
		\label{eq:strivial}
		\lim_{n \to \infty}	d_{\mathrm{TV}}\left( (K_1, \ldots, K_{\min(N_n, \lfloor n / \mu\rfloor - s_n)}), (X_1, \ldots, X_{\lfloor n/ \mu\rfloor - s_n} ) \right)  = 0.
	\end{align}
	In specific cases asymptotic upper bounds for $s_n$ may be made explicit:
	\begin{enumerate}[\qquad a)]
		\item 
		If $\Ex{X^3}<\infty$, or if $\Ex{X^2}< \infty$ and $\Ex{X^2 \one_{X \ge n}} = O(n^{-1/3})$, then~\eqref{eq:strivial} holds for any sequence $s_n = o(n)$ satisfying $s_n / n^{3/4} \to \infty$.
		\item If $\Ex{X^2}< \infty$ and $\Ex{X^2 \one_{X \ge n}} = O(n^{-r})$ for some $0<r<1/3$, then ~\eqref{eq:strivial} holds for any sequence $s_n = o(n)$ satisfying $s_n / n^{\frac{1}{1+r}} \to \infty$.
	\end{enumerate}
\end{lemma}
The conditions in cases a) and b) are related to rates of convergence in local limit theorems~\cite{MR0322926,zbMATH03504209,zbMATH06778287}. Lemma~\ref{le:independent} implies that any functional $F(K_1, \ldots, K_{N_n})$ that typically does not get perturbed by the last $o(n)$ coordinates behaves in the same way as the functional $F(X_1, \ldots, X_{\lfloor n / \mu \rfloor})$. This has far reaching consequences, as it allows us to transfer a multitude of asymptotics for i.i.d. variates to the Gibbs partition model with little effort. 

As an example, we describe applications to order statistics.

\begin{corollary}
	\label{co:maxcomp}
	Let $X_{(1)} \ge X_{(2)} \ge  \ldots$ denote the random variables $X_1, \ldots, X_{\lfloor n / \mu \rfloor}$  arranged in decreasing order. Suppose that one of the following two conditions are met:
	\begin{enumerate}[i)]
		\item Assumption i) of Theorem~\ref{te:dense} is satisfied.
		\item Assumption ii) of Theorem~\ref{te:dense} is satisfied, and there exists $A>0$ such that for each $\epsilon>0$ there exists a sequence $(\kappa(n))_{n \ge 1}$ with $\epsilon/ A \le \Pr{X\ge \kappa(n)} n \le \epsilon A$ for large enough $n$. For example, this is always the case when $\lim_{n \to \infty} \frac{w_{n+1}}{w_n} \in ]0, \infty[$.
	\end{enumerate}
	Then  there exists a sequence $t_n \to \infty$ such that
	\begin{align}
		\label{eq:megatv}
		\lim_{n \to \infty} d_{\mathrm{TV}}\left( (K_{(1)}, \ldots, K_{(\min(N_n,t_n))}), (X_{(1)}, \ldots, X_{(t_n)})  \right) = 0.
	\end{align}
	In specific cases asymptotic lower bounds for $t_n$ may be made explicit:
	\begin{enumerate}[\qquad a)]
		\item 
		If $\Ex{X^3}<\infty$, or if $\Ex{X^2}< \infty$ and $\Ex{X^2 \one_{X \ge n}} = O(n^{-1/3})$, then~\eqref{eq:megatv} holds for any sequence $t_n$ satisfying $t_n = o(\frac{n^{1/4}}{\log n})$.
		\item If $\Ex{X^2}< \infty$ and $\Ex{X^2 \one_{X \ge n}} = O(n^{-r})$ for some $0<r<1/3$, then ~\eqref{eq:megatv} holds for any sequence $t_n$ satisfying $t_n = o(\frac{n^{\frac{r}{1+r}}}{\log n})$.
	\end{enumerate}
\end{corollary}

Extreme values of i.i.d. variates are a well-understood subject~\cite{zbMATH05242364, zbMATH03820796}, and Corollary~\ref{co:maxcomp} ensures that we may transfer distributional results of all kinds from this setting to the Gibbs partition model. In particular, we may transfer such results from~\cite[Thms. 19.7, 19.16, 19.19, Cor. 19.14]{MR2908619}. To be precise,~\cite{MR2908619} studies a model whose extremal sizes may in specific settings be approximated either in total variation or with respect to the Kolmogorov distances by extremal sizes of i.i.d. variates~\cite[Thm. 19.7]{MR2908619}. But this yields in turn that these results also hold for extremal sizes of i.i.d. variates and by Corollary~\ref{co:maxcomp} hence also for extremal sizes of Gibbs partitions:

\begin{corollary}
	\label{co:profile}
	Under the assumptions of Corollary~\ref{co:maxcomp}, and assuming that $w_n>0$ for sufficiently large $n$, the following statements hold:
	\begin{enumerate}
		\item Suppose that $\Ex{X^2}<\infty$. Let $(\kappa(n))_{n \ge 1}$ denote a sequence such that  $\lambda_0 := \lim_{n \to \infty} n\Pr{X \ge \kappa(n)} \in [0, \infty]$. Then
		\[
		|\{1 \le i \le N_n \mid K_i \ge \kappa(\lfloor n/\mu \rfloor)\}| \convdis \mathrm{Poi}(\lambda_0).
		\]
		Here $\mathrm{Poi}(\lambda_0) \ge 0$ denotes a Poisson random variable with parameter $\lambda_0$, and degenerate cases $\mathrm{Poi}(0) = 0$ and $\mathrm{Poi}(\infty) = \infty$. Furthermore, for any integer $j \ge 1$ we have
		\[
		\lim_{n \to \infty} \sup_{k \ge 0} \left| \Pr{K_{(j)} \le k } - \Pr{ \mathrm{Poi}(\lfloor n/\mu \rfloor\Pr{X>k}) <j  } \right| = 0.
		\]
		\item Suppose that $X$ is light-tailed and let $k(n) = \sup\{ k \ge 1 \mid \Pr{X=k} \ge 1/n \}$. If $\lambda_1 := \lim_{n \to \infty}\frac{\Pr{X = n+1}}{\Pr{X=n}} \in [0,1[$ exists, then  for each integer $j \ge 1$
		\[
		K_{(j)}= k(\lfloor n/\mu \rfloor) + O_p(1).
		\]
		If $\lambda_1 = 0$, then we  even have
		\[
		\lim_{n \to \infty} \Pr{ |K_{(j)} - k(\lfloor n/\mu \rfloor)| \le 1} = 0.
		\]
		If $\lambda_1 > 0$, then
		\begin{align*}
			K_{(1)} &\atv \lfloor G   / \log(1/\lambda_1) + \log_{1/\lambda_1}K(\lfloor n/\mu \rfloor)  \rfloor, \\
			K_{(1)} - k(\lfloor n/\mu \rfloor) &\atv \left\lfloor G  + \log\left(\frac{  \lfloor n/\mu \rfloor \Pr{X=k(\lfloor n/\mu \rfloor)}}{1-\lambda_1}\right)\frac{1}{\log(1/\lambda_1)} \right\rfloor
		\end{align*}
		for a random variable $G$ with Gumbel distribution
		\[
		\Pr{G \le x} = \exp(- \exp(-x)), \qquad x \in \ndR
		\]
		and
		\[
		K(n) = \frac{ n \Pr{X=k(n)} \lambda_1^{-k(n)} }{ 1 - \lambda_1 }.
		\]
	\end{enumerate}
\end{corollary}

This complements Theorem~\ref{te:functional}, which treats the infinite variance case for $1 < \alpha < 2$. We remark that the statistics in~Corollary~\ref{co:jump} for this case may alternatively be deduced  from~Corollary~\ref{co:maxcomp}.

Lemma~\ref{le:independent} also ensures that the number $\#_k P_n$ of components of size $k \ge 0$ behaves like a $\mathrm{Bin}(\lfloor n / \mu \rfloor, \Pr{X=k})$ random variable:

\begin{corollary}
	\label{co:comptheifrst}
	Suppose that the assumptions of Theorem~\ref{te:dense} hold. If $(k_n)_{n \ge 1}$ is a sequence of positive integers with $n \Pr{X=k_n} \to \infty$, then for any $\epsilon>0$
	\[
	\lim_{n \to \infty} \Prb{ \left|\frac{\#_{k_n} P_n}{\Pr{X=k_n} n\mu^{-1}} -1\right| > \epsilon} = 0.
	\]
	If $(k_n)_{n \ge 1}$ satisfies $(n/\mu) \Pr{X=k_n} \to \lambda_2 \in [0, \infty[$, then 
	\[
	\#_{k_n} P_n \convdis \mathrm{Poi}(\lambda_2).
	\]
\end{corollary}

Furthermore, we may clarify connections to the balls in boxes model:

\begin{remark}
	By Kolchin's representation theorem (see Proposition~\ref{pro:kolchin} below) we know that if $\rho_u > 0$ then for any integer $\ell$ with $\Pr{N_n = \ell} >0$ we have
	\begin{align}
		\label{eq:bnb}
		\left( (K_1, \ldots, K_{N_n}) \mid N_n = \ell \right) \eqdist \left( (X_1, \ldots, X_\ell) \mid X_1 + \ldots + X_\ell = n    \right).
	\end{align}
	The model on the right-hand side is called a balls in boxes model, with $n$ balls and $\ell$ ordered boxes. It has been extensively studied in~\cite{MR2908619}, in particular for many cases where the number of balls is proportional to the number of boxes with fluctuations  of smaller order. Equation~\eqref{eq:bnb} tells us that the component sizes of a Gibbs partition are distributed like a balls in boxes model with a random number of boxes. By  Theorem~\ref{te:dense} we know that this random number concentrates at a constant multiple of the number of balls, with precise information on the fluctuations. In some cases, results for balls in boxes models hold uniformly for fluctuations of sufficiently high order, allowing us to transfer them via Equation~\eqref{eq:bnb} to the Gibbs partition model. See for example~\cite[Thm. 19.7, first phrase after condition (vi)]{MR2908619}. Many results in~\cite{MR2908619} assume, however,  fluctuations of smaller order than those exhibited by $N_n$. Hence the need for Corollary~\ref{co:maxcomp}. See also~\cite[Problem 19.10]{MR2908619}.
\end{remark}

As a final remark, we note that although Lemma~\ref{le:independent} guarantees that most components become independent from each other, they do not become independent from~$N_n$. The best we can do when including $N_n$ is a contiguity relation:

\begin{lemma}
	\label{le:contiguity}
	Under the assumptions of Theorem~\ref{te:dense}, for each $0 < \delta < \mu^{-1}$ and each $\epsilon \in ]0,1[$ there exists a sequence $(\cF_n)_{n \ge 1}$ and constants $c,C>0$ such that for all sufficiently large~$n$
	\begin{align}
		\Pr{(N_n, K_1, \ldots, K_{N_n - \lfloor\delta n\rfloor)}) \in \cF_n} > 1 - \epsilon
	\end{align}
	and
	\begin{align}
		\Pr{ (N_n, X_1, \ldots, X_{N_n - \lfloor\delta n\rfloor} ) \in \cF_n} > 1 - \epsilon
	\end{align}
	and uniformly for all elements $F \in \cF_n$
	\begin{align}
		c < \frac{\Pr{(N_n, K_1, \ldots, K_{N_n - \lfloor\delta n\rfloor}) = F}}{\Pr{ (N_n, X_1, \ldots, X_{N_n - \lfloor\delta n\rfloor} ) = F}} < C.
	\end{align}
\end{lemma}

Contiguity relations of this form are very useful for combinatorial applications, see for example~\cite{2022arXiv220200592S}.

\subsection{The convergent case}
\label{sec:convergent}

The convergent case is characterized by a stochastically bounded number of components in the Gibbs partition $P_n$, with a unique giant component and a finite distributional limit for the small fragments.

If $0<W(\rho_w) < \infty$, we may define a random variable $X$ as in Equation~\eqref{eq:defX}, with probability generating function
\[
\Ex{z^X} = \frac{W(\rho_w z)}{W(\rho_w)}.
\]
We let $(X_i)_{i \ge 1}$ denote independent copies of $X$.
If additionally $0<V'(W(\rho_w))<~\infty$, we may define  random variables $N$ and ~$\hat{N}$ with probability generating functions
\begin{align}
	\Ex{z^N} = \frac{V(W(\rho_w)z) }{V(W(\rho_w))} \qquad \text{and} \qquad \Ex{z^{\hat{N}}} = \frac{V'(W(\rho_w)z) z}{V'(W(\rho_w))},
\end{align}
so that $N, \hat{N}$ are independent from~$(X_i)_{i \ge 1}$.
Let $S_{\hat{N}-1} = \sum_{i=1}^{\hat{N}-1}X_i$ and let $J \in \{1, \ldots, \hat{N}-1\}$ be drawn uniformly at random, independent from all previously considered random variables. We let
\begin{align}
	\ndN^+ = \bigcup_{n \ge 1} \ndN_0^n
\end{align}
the collection of all finite non-empty sequences of non-negative integers. This allows us to form the random finite tuple of integers
\begin{align}
	\Xi_n = \left(X_1, \ldots, X_{J-1}, n - S_{\hat{N}-1}, X_{J}, \ldots, X_{\hat{N}-1}\right).
\end{align}

\begin{theorem}
	\label{te:convergent}
	Suppose that $0<V'(W(\rho_w))< \infty$.
	Each of the following additional conditions is sufficient in order for 
	\begin{align}
		\label{eq:convcomponents}
		N_n \convdis \hat{N}
	\end{align}
	and
	\begin{align}
		\label{eq:tvapprox}
		\lim_{n \to \infty} \sup_{A \subset \ndN^+} \left| \Prb{(K_1, \ldots, K_{N_n}) \in A} - \Prb{\Xi_n \in A}  \right| = 0
	\end{align}
to hold.
	\begin{enumerate}[\qquad i)]
		\item We have $\rho_v = W(\rho_w)$ and \[
		v_n = L_v(n) n^{-b}\rho_v^{-n} \qquad \text{and} \qquad w_n = L_w(n) n^{- a} \rho_w^{-n}
		\]
		for slowly varying functions~$L_v$ and $L_w$, an exponent $b > 2$, and an exponent~$a$ such that either $1 < a < b$, or $a=b$ and $L_v(n) = o(L_w(n))$.
		\item We have $\infty \ge \rho_v > W(\rho_w)$ and \[
		\frac{w_n}{w_{n+1}} \to \rho_w > 0 \qquad \text{and} \qquad \frac{1}{w_n} \sum_{i+j = n} w_i w_j \to 2 W(\rho_w) < \infty\] as $n \to \infty$. (Note that the last two limits are automatically satisfied in case $w_n = L_w(n) n^{- a} \rho_w^{-n}$ with $a>1$ constant,  $L_w$ slowly varying, and $\rho_w > 0$.)
		\item We have $\Ex{N^{1+a+\delta}}< \infty$ and $w_n = L_w(n) n^{- a} \rho_w^{-n}$ for $a>1$, $\delta>0$ constants  and $L_w$ a slowly varying function.
		\item We have $w_n \sim c_w n^{- a} \rho_w^{-n}$ for constants $a>1$, $c_w > 0$, and additionally one of the following conditions hold.
		\subitem a) $\Ex{N^{1+a}}< \infty$.
		\subitem b) $1<a<2$.
		\subitem c) $a = 2$ and $\Ex{N (\log N)^{2 + \delta}} < \infty$ for some $\delta >0$.
		\subitem d) $2<a<3$ and $\Pr{N=n} = o(\Pr{X=n})$.
		\subitem e) $a=3$ and $\Pr{N = n} = o(1 / (n^3 \log \log n))$.
	\end{enumerate}
\end{theorem}

Under the assumptions of Theorem~\ref{te:convergent} it follows that the Gibbs partition $P_n$ has a unique giant component whose size $M_n$ satisfies
\begin{align}
	n - M_n  \convdis  \sum_{i=1}^{\hat{N}-1} X_i.
\end{align}
It also follows that the result $\phi(K_1, \ldots, K_{N_n})$ of replacing the left-most maximal component by placeholder value $*$ satisfies
\begin{align}
	\label{eq:convergentrephrase}
	\phi(K_1, \ldots, K_{N_n}) \convdis \left(X_1, \ldots, X_{J-1}, *, X_{J}, \ldots, X_{\hat{N}-1}\right).
\end{align}

Asymptotic behaviour of this form was proven in~\cite{MR2121024} for random partitions satisfying a conditioning relation using a perturbed Stein recursion approach.
Theorem~\ref{te:convergent} on Gibbs partitions was verified later for condition ii) in~\cite{zbMATH06964621}, using the theory of subexponential probability distributions~\cite{MR0348393, MR714482,MR772907,MR3097424}. The wealth of additional sufficient conditions that even include cases of critical composition schemes is made possible by  new results on local probabilities of randomly stopped sums by~\cite{zbMATH07179510}. 

In~\cite[Thm. 1]{zbMATH01135939}  the size of the largest component of $P_n$ was determined assuming that $\infty \ge \rho_v > W(\rho_w)$ and that the function $W(z)$ admits a singularity expansion at its dominant singularity that results in $w_n \sim c_w \rho_w^{-n} n^{-a} (\log n)^\lambda$ for $1<a< 2$ and $\lambda \in \ndR$. Furthermore, \cite[Prop. IX.1]{MR2483235} determined the asymptotic number of components under similar analytic assumptions with $\lambda=0$. Theorem~\ref{te:convergent} encompasses these settings  and additionally shows that there is a  limit distribution for the small fragments.

\subsection{The mixture case}

\label{sec:mixture}

In the mixture case, with a  limiting probability $p \in ]0,1[$ the Gibbs partition $P_n$ behaves as in the dense case, and with limiting probability $1-p$ it behaves as in the convergent case.

We use the notation from Sections~\ref{sec:dense} and~\ref{sec:convergent}.
\begin{theorem}
	\label{te:mixture}
	Suppose that $\rho_v = W(\rho_w)$. Furthermore, suppose that
	\[
	v_n = L_v(n) n^{-a}\rho_v^{-n} \qquad \text{and} \qquad w_n = L_w(n) n^{- a} \rho_w^{-n}
	\]	
	for slowly varying functions~$L_v, L_w$ and a constant $a > 2$. Suppose that the limit
	\[
	p := \lim_{n \to \infty} \frac{\mu^{a-1}}{V'(W(\rho_w))}  \frac{L_v(n)}{L_w(n)} >0
	\]
	exists and is positive. Then the event $\cE_n= \{ N_n \ge n / (2 \mu)\}$ satisfies
	\begin{align}
		\lim_{n \to \infty} \Pr{\cE_n} = p \in ]0,1[
	\end{align}
	and the following conditional properties.
	\begin{enumerate}[\qquad a)]
		\item On the event $\cE_n$, the Gibbs partition behaves as in the dense regime:
		\begin{align}
			\label{eq:mixtureclt}
			\left( \frac{N_n - n/\mu}{L(n) n^{1/\alpha}} \,\,\Big\vert\,\, \cE_n \right) \convdis X_\alpha(\gamma, -1, 0)
		\end{align}
		as $n \to \infty$,  and for any $\delta > 0$
		\begin{align}
			\label{eq:mixturellt}
			\lim_{n \to \infty} \sup_{\ell \ge \delta n} \left| L(n) n^{1/\alpha} \Pr{ N_n = \ell \mid \cE_n} - h\left( \frac{\ell - n/\mu}{L(n) n^{1/\alpha}}  \right)  \right| =0,
		\end{align}
		with the index $\ell$ ranging over all integers satisfying $(\ell - \mu n)/(L(n) n^{1/\alpha}) \in \Omega$. 
		\item On the complementary event $\cE_n^c$, the Gibbs partition behaves as in the convergent regime:
		\begin{align}
			\label{eq:mixturecomponents}
			(N_n \mid \cE_n^c) \convdis \hat{N}
		\end{align}
		and
		\begin{align}
			\label{eq:mixturetvapprox}
			\lim_{n \to \infty} \sup_{A \subset \ndN^+} \left| \Prb{(K_1, \ldots, K_{N_n}) \in A \mid \cE_n^c} - \Prb{\Xi_n \in A}  \right| = 0.
		\end{align}
	\end{enumerate}
\end{theorem}

Such a mixture behaviour for the number of components was previously observed for the case $\alpha=3/2$ in the famous work~\cite[Thm. 5]{MR1871555}, assuming that the generating series $V(z)$ and $W(z)$ admit singularity expansions that result in asymptotic  power laws with exponent $-5/2$ for their coefficients. The proofs in~\cite{MR1871555} are based on singularity analysis and saddle-point bounds. Here we pursue a probabilistic approach using local limit theorems for random walks and randomly stopped sums~\cite{zbMATH07179510} instead.

The statements of Theorem~\ref{te:functional}, Corollary~\ref{co:jump}, Lemma~\ref{le:independent}, Corollary~\ref{co:maxcomp}, and Lemma~\ref{le:contiguity} hold by analogous proofs for $((K_1, \ldots, K_{N_n}) \mid \cE_n)$. In particular, under the assumptions of Theorem~\ref{te:mixture}, we have
\begin{align}
	\left( \left( \frac{\sum_{i=1}^{\lfloor s N_n \rfloor}K_i - N_n s \mu }{L(n)n^{1/\alpha}}, \,\, 0 \le s \le 1 \right)  \,\,\Big\vert\,\, \cE_n \right) \convdis (\mu Y_s, \,\, 0 \le s \le 1)
\end{align}
in $\mathbb{D}([0,1], \ndR)$ as $n \to \infty$, and consequently
\begin{align}
	\left( \frac{1}{L(n)n^{1/\alpha}} \max_{1 \le i \le N_n} K_i \,\,\Big\vert\,\, \cE_n \right) \convdis \begin{cases}W_1, &1<\alpha<2 \\ 0, &\alpha=2. \end{cases}
\end{align}
Likewise, for each integer $j \ge 2$
\begin{align}
	\frac{1}{L(n)n^{1/\alpha}} (K_{(j)} \mid \cE_n) \convdis W_j
\end{align}
in case $1 < \alpha < 2$.

\subsection{The dilute case}

\label{sec:dilute}

In the dilute case the number of components gets spread out, with a  continuous limit distribution after rescaling without centring by $n^{-\alpha}$ for $0< \alpha < 1$.

\begin{theorem}
	\label{te:dilute}
	Suppose that $\rho_v = W(\rho_w)$. Furthermore, suppose that
	\[
	v_n = L_v(n) n^{-b}\rho_v^{-n} \qquad \text{and} \qquad w_n \sim c_w n^{- a} \rho_w^{-n}
	\]	
	for a slowly varying function~$L_v$, a constant $c_w>0$, and exponents $1<a,b<2$. Let $\alpha=a-1$, and let $f$ denote the density function of the stable distribution $S_\alpha(\gamma, 1,0)$ for
	\[
	\gamma= \left(\frac{c_w}{W(\rho_w)\alpha} \Gamma(1- \alpha) \cos \frac{\pi \alpha}{2} \right)^{1/\alpha}.
	\]
	Then
	\begin{align}
		\label{eq:diluteclt}
		\frac{N_n}{n^{\alpha}} \convdis Z
	\end{align}
	for a random variable $Z>0$ with density function
	\[
	\tilde{f}(x) = \frac{1}{\alpha \Ex{(X_{\alpha}(\gamma, 1,0))^{\alpha(b-1)}}} \frac{f\left(\frac{1}{x^{1/\alpha}}\right)}{x^{b+1/\alpha}}.
	\]
	For any constant $\delta>0$ we have a local limit theorem
	\begin{align}
		\label{eq:dilutellt}
		\lim_{n \to \infty} \sup_{\ell \ge \delta n^{\alpha}} \left|n^{\alpha} \Pr{ N_n = \ell} - \tilde{f}(\ell / n^{\alpha})  \right| =0.
	\end{align}
\end{theorem}
A similar local limit theorem for the number of components  as in Theorem~\ref{te:dilute}, and similar component size statistics as in Corollary~\ref{co:poiz} below were  recently established using analytic methods in~\cite[Thm. 4.1, Thm. 5.1]{2021arXiv210303751B}  alongside a multitude of combinatorial applications,  assuming that the generating series $V(z)$ and $W(z)$ admit suitable singular expansions. The approach for Theorem~\ref{te:dilute} and its applications is instead based on local probabilities of random walks and randomly stopped sums~\cite{zbMATH07179510}. The work~\cite{2021arXiv210303751B} also studies limits of extended composition schemes, which we treat in Section~\ref{sec:product} below using a unified approach that encompasses all combinations of regimes  (dense, convergent, mixture, and dilute) under consideration.

By~\cite[Thm. 4.1]{2011arXiv1112.0220J}, the density function $f$ satisfies for all $x>0$
\begin{align}
\lambda^{1/\alpha}f(x\lambda^{1/\alpha}) =  \frac{1}{\pi x} \sum_{k=1}^\infty \frac{\Gamma(k \alpha +1)}{k!} (-x^{-\alpha})^k \sin(- \alpha k \pi),
\end{align}
with
\begin{align}
\label{eq:lambdadilute}	
\lambda := \frac{c_w}{W(\rho_w)\alpha} \Gamma(1-\alpha).
\end{align}
Furthermore, by~\cite[Thm. 5.1, Ex. 5.5]{2011arXiv1112.0220J} it holds for all $s \in \ndC$ with $\Re(s) < \alpha$ that
\begin{align}
	\label{eq:xalphadilute}
	\Ex{X_{\alpha}(\gamma, 1,0)^s} = \lambda^{s/\alpha} \frac{\Gamma(1- s/\alpha)}{\Gamma(1-s)}.
\end{align}
Using integration by substitution it follows that for all $r \in \ndC$ with $\Re(r) > b-2$
\begin{align}
	\label{eq:momzr}
	\Ex{Z^r} = \frac{\Ex{(X_{\alpha}(\gamma, 1,0))^{\alpha(b-r-1)}}}{\Ex{(X_{\alpha}(\gamma, 1,0))^{\alpha(b-1)}}} = \lambda^{-r} \frac{\Gamma(2-b+r) \Gamma(1- \alpha(b-1))}{\Gamma(2-b) \Gamma(1- \alpha(b-r-1))}.
\end{align}
Furthermore, it follows that the density function $\tilde{f}$ of $Z>0$ admits the expression
\begin{align}
	\tilde{f}(x) = \frac{\Gamma(1- \alpha(b-1))}{\Gamma(2-b) \alpha^{2-b} \left( \frac{c_w}{W(\rho_w)} \Gamma(1-\alpha) \right)^{b-1}} \frac{f\left(\frac{1}{x^{1/\alpha}}\right)}{x^{b+1/\alpha}}.
\end{align}

As before, we let $X$ denote a random non-negative integer with probability generating function
$
\Ex{z^X} = \frac{W(\rho_w z)}{W(\rho_w)},
$
so that $\Pr{X=n} \sim \frac{c_w}{W(\rho_w)}n^{-a} $.

\begin{corollary}
	\label{co:poiz}
	Suppose that the assumptions of Theorem~\ref{te:dilute} hold.  If $(k_n)_{n \ge 1}$ is a sequence of positive integers with $n^\alpha \Pr{X=k_n} \to \infty$, then 
	\begin{align}
		\label{eq:refer1}
	\frac{\#_{k_n} P_n}{\Pr{X=k_n}n^{\alpha}} \convdis Z.
	\end{align}
	If $(k_n)_{n \ge 1}$ satisfies $n^{\alpha} \Pr{X=k_n} \to \upsilon \in [0, \infty[$ instead (that is, if $k_n \sim \left( \frac{c_w}{\upsilon W(\rho_w)} \right)^{\frac{1}{1+\alpha}} n^{\frac{\alpha}{1+\alpha}}$ for $\upsilon>0$, and $k_n \gg n^{\frac{\alpha}{1+\alpha}}$ for $\upsilon=0$), then 
	\begin{align}
		\label{eq:refer2}
	\#_{k_n} P_n \convdis \mathrm{Poi}(\upsilon Z).
	\end{align}
	Here $\mathrm{Poi}(\upsilon Z)$ denotes a Poisson random variable with  random parameter $\upsilon Z$, so that
	\[
		\Pr{\mathrm{Poi}(\upsilon Z) = k} = \Exb{ \frac{(\upsilon Z)^k}{k! \exp(\upsilon Z)}}, \qquad k \ge 0.
	\]
\end{corollary}

In the case $n^{\alpha} \Pr{X=k_n} \to \upsilon \in [0, \infty[$, the fact that $\Pr{X=n+1} \sim \Pr{X=n}$ allows us to easily extend the proof of Corollary~\ref{co:poiz} to show that for any finite set $M$ we have 
\begin{align}
\sum_{k \in k_n + M} \#_{k} P_n \convdis \mathrm{Poi}(\upsilon Z  |M|).
\end{align}

Corollary~\ref{co:poiz} entails that when $k_n \gg n^{\frac{\alpha}{1+ \alpha}}$, then $\#_{k_n} P_n \convp 0$. Nevertheless, components  whose size has larger order than  $n^{\frac{\alpha}{1+ \alpha}}$ are still likely to exist. For example, using the Chernoff bounds analogously as for Equation~\eqref{eq:arganalougous} yields that at least  in the case $\frac{n^\alpha \Pr{X \ge k_n}}{\log n} \to \infty$ (which is  equivalent to $\frac{n}{k_n (\log n)^{1/\alpha}} \to \infty$) we have  $\frac{\sum_{k \ge k_n} \#_k P_n}{N_n \Pr{X \ge k_n}} \convp 1$. Thus, by Theorem~\ref{te:dilute} and Slutsky's theorem, we have  in this case
\begin{align}
	\label{eq:sumkconvdisz}
	\frac{\sum_{k \ge k_n} \#_k P_n}{n^\alpha \Pr{X \ge k_n}} \convdis Z.
\end{align}
Consequently, when $k_n'$ is a sequence with $k_n \ll k_n'$ then the contribution of components with size larger than $k_n'$ is negligible in this sum, and $\frac{\sum_{k = k_n}^{k_n'} \#_k P_n}{n^\alpha \Pr{X \ge k_n}} \convdis Z.$
Since $Z>0$ this guarantees the existence of a large number of components whose size is close to $n$.  
The following observation shows us that the maximal component size $K_{(1)} =\max(K_1, \ldots, K_{N_n}) \le n$ indeed has order $n$ and determines the limiting law of the fluctuations. In fact, we obtain a point process limit. 

\begin{corollary}
	\label{co:pointprocess}
	Suppose that the assumptions of Theorem~\ref{te:dilute} hold. Define the following point process on $]0,1]$
		\[
			\Upsilon_n = \sum_{\substack{1 \le i \le N_n \\ K_i > 0}} \delta_{K_i / n},
		\]
		with $\delta$ referring to the Dirac measure. Then
	\begin{align}
		\Upsilon_n \convdis \Upsilon
	\end{align}
	as $n \to \infty$, for a point process $\Upsilon$ on $]0,1]$ with intensity 
	\[
     \frac{x^{-\alpha-1} (1-x)^{\alpha(2-b)-1}}{B(1-\alpha, \alpha(2-b))  }\,\mathrm{d}x.
	\]
	Here $B(\cdot, \cdot)$ denotes Euler's beta function. Almost surely, $\Upsilon$ has infinitely many points.
\end{corollary}
\noindent It is remarkable that the limiting process only depends on $\alpha$ and $b$, and not on the asymptotic constant $\frac{c_w}{W(\rho_w)}$. In the proof we use similar arguments as for the result~\cite[Ex. 19.27]{MR2908619} on the balls-in-boxes model with weights belonging to the domain of attraction of a stable law with index in the interval $]1,2[$. In Equation~\eqref{eq:wellitried}, we obtain for any integer $m \ge 1$ and any $0<x\le1$ an expression for the $m$-th factorial moment  $\Ex{(\Upsilon([x,1]))_m}$, which by Equations~\eqref{eq:lambdadilute},~\eqref{eq:xalphadilute} and~\eqref{eq:momzr} and a short calculation simplifies to
\begin{align}
	\Ex{(\Upsilon([x,1]))_m} = \Exb{ \left(\mathrm{Poi}\left( \frac{c_w}{W(\rho_w)}Z\right)\right)_m } \int_x^1 \cdots \int_x^1 \one_{y_1 + \ldots + y_m \le 1} \frac{(1-y_1 - \ldots - y_m)^{\alpha(m+1-b)-1}}{y_1^{\alpha+1}\cdots y_m^{\alpha+1}  } \,\mathrm{d}y_1 \cdots \mathrm{d}y_m,
\end{align}
with
\begin{align*}
	\Exb{ \left(\mathrm{Poi}\left( \frac{c_w}{W(\rho_w)}Z\right)\right)_m } 
	= \left( \frac{ \alpha } {\Gamma(1- \alpha)} \right)^m  \frac{\Gamma(1 + \alpha(1-b)) \Gamma(m + 2-b)}{\Gamma(1 + \alpha(m+1-b)) \Gamma(2-b)}.
\end{align*}

We let $\eta_1 \ge \eta_2 \ge \ldots>0$ denote the ranked points of $\Upsilon$ on $]0,1]$. Since almost surely $\Upsilon(]0,1])=\infty$, there are indeed infinitely many such points.
 By similar arguments as in~\cite[Lem. 4.4]{zbMATH01877115} the point process limit in Corollary~\ref{co:pointprocess} entails that for all integers $k \ge 1$ the size $K_{(k)}$ of the $k$th largest component of $P_n$ satisfies the joint distributional convergence
\begin{align}
	K_{(k)} / n \convdis \eta_k, \qquad k \ge 1 \text{ (jointly)}.
\end{align}
 We have almost surely
\[
\Upsilon([x,1]) \le 1/x,
\]
since the sum of the component sizes in $P_n$ equals $n$. Thus
\[
\eta_k \le 1/k
\]
for all $k \ge 1$. Thus the probability generating function $\Ex{z^{\Upsilon([x,1])}}$ is a polynomial in $z$ of degree at most $\lfloor 1/x \rfloor$, given by
\[
	\Exb{z^{\Upsilon([x,1])}} = \sum_{m=0}^{\lfloor 1/x \rfloor} \frac{1}{m!} \Exb{(\Upsilon([x,1]))_m} (z-1)^m.
\]
The distribution of the ranked point $\eta_k$ is determined by 
\[
\Pr{\eta_k < x} = \Pr{\Upsilon([x,1]) < k} = \sum_{j=0}^{k-1}\frac{1}{j!} \frac{\mathrm{d}^j}{\mathrm{d} z^j} \Ex{z^{\Upsilon([x,1])}} \bigg\vert_{z=0}
\]
for all $0< x \le 1$. 
The integrals in these formulas may be simplified: 

\begin{proposition}
	\label{pro:explicit}
	For all $0<x\le 1$ and all integers $k \ge 1$ we have
	\begin{multline}
		\label{eq:etakkk}
		\Pr{\eta_k < x} =  \frac{ \Gamma(1- \alpha(b-1)) }{ 2\pi |\Gamma(-\alpha)|^{b-1} \Gamma(2-b)}  \\\int_0^\infty \int_{-\infty}^\infty   u^{\alpha(b-1)} \exp\left( -\frac{1}{u^\alpha}\int_I \frac{e^{ityu}}{y^{1+\alpha}}\,\mathrm{d}y -iut -  |\Gamma(-\alpha)| (-it)^\alpha\right)   \sum_{j=0}^{k-1}\frac{1}{j!} \left( \frac{1}{u^\alpha}\int_I \frac{e^{ityu}}{y^{1+\alpha}} \,\mathrm{d}y \right)^j  \,\mathrm{d}t\,\mathrm{d}u.
	\end{multline}
Furthermore, we have almost surely $0< \eta_k <1$. 
\end{proposition}
The proof is detailed in Section~\ref{sec:proofdilute}, alongside the other proofs for the dilute regime. The calculations  are similar to those in~\cite{zbMATH00563567} and~\cite[Rem. 19.28]{MR2908619}. Note that we may not change the order of integration  of the two outer integrals in this expression.

As a final remark, if $k_n$ is an arbitrary  sequence of integers satisfying $n^{\frac{\alpha}{1+\alpha}} \ll k_n \ll n$, then by Equation~\eqref{eq:intermediate} below we may for each constant $C>0$ choose a sequence $k_n' \sim k_n$ of integers with $(n/k_n)^\alpha - (n/k_n')^\alpha \to C$ and
\begin{align}
	\sum_{k=k_n}^{k_n'} \#_k P_n \convdis \mathrm{Poi}\left(\frac{C}{\alpha} \frac{c_w}{W(\rho_w)} Z\right).
\end{align}
Together with Equation~\eqref{eq:sumkconvdisz} and subsequent remarks, this describes the component sizes between the scales $n^{\frac{\alpha}{1+\alpha}}$ and $n$.

\subsection{Product structures and extended composition schemes}

\label{sec:product}

Under the assumptions of Theorem~\ref{sec:convergent} the Gibbs partition of the  composition schema $U(z) = W(z)^2$ admits a mixture limiting behaviour: with asymptotic probability $1/2$ the first coordinate is large and the second converges to a finite limit distribution, and on the complementary event the second is large and the first converges to a finite limit. This may be generalized to the asymptotic behaviour of product structures with partition function generating series of the form 
\begin{align}
	O(z) = \prod_{k=1}^\ell W_k(z)
\end{align} for some $\ell \ge 2$ and generating series $W_k(z)$, $1 \le k \le \ell$. We let $\rho_o$ denote the radius of convergence of $O(z)$. For any integer $n \ge 0$ with $o_n = [z^n] O(z)>0$ we may consider the random vector
$(P_1, \ldots, P_\ell)$ of non-negative integers with probability generating function
\begin{align}
	\Ex{z_1^{P_1} \cdots z_\ell^{P_\ell}} = \frac{ [z^n] \prod_{k=1}^\ell W_k( z_kz) }{o_n }.
\end{align}

For ease of notation, we say a sequence $(r_n)_{n \ge 0}$ of non-negative real numbers satisfies the \emph{subexponentiality} condition, if the associated power series $r(z) = \sum_{n \ge 0} r_n z^n$ has a radius of convergence $\rho_r \in ]0, \infty[$, and if
\[
\frac{r_n}{r_{n+1}} \to \rho_r \qquad \text{and} \qquad \frac{1}{r_n} \sum_{i+j=n} r_i r_j \to 2 r(\rho_r)<\infty
\]
as $n \to \infty$. Note that these conditions are automatically satisfied if $r_n = L_r(n) n^{-c}\rho_r^{-n}$ for constants $c>1$, $\rho_r>0$, and a slowly varying function $L_r$.

\begin{lemma}
	\label{te:extended}
	Suppose that there is a sequence $(r_n)_{n \ge 0}$ satisfying the subexponentiality condition, such that the limit
	\begin{align*}
		\lim_{n \to \infty} \frac{w_n^{(k)}}{r_n} \in [0, \infty[
	\end{align*}
	exists for each $1 \le k \le \ell$, and is positive for at least one $1 \le k \le \ell$. Then $(o_n)_{n \ge 0}$ satisfies the subexponentiality condition and the limit
	\[
	p_k = \lim_{n \to \infty} \frac{w_n^{(k)} O(\rho_o)}{o_n W_k(\rho_o)} \in [0,1]
	\]
	exists for all $1 \le \ell \le k$. The limiting constants $p_1, \ldots, p_\ell$ satisfy $\sum_{k=1}^\ell p_k = 1$, allowing us to define a random integer $1 \le I \le \ell$ with
	\[
	\Pr{I = k} = p_k, \qquad 1 \le k \le \ell.
	\]
	For each $1 \le k \le \ell$ let $X^{(k)}$ denote the random non-negative integer with probability generating function
	\[
	\Exb{z^{X^{(k)}}} = \frac{W_k(z\rho_w^{(k)}  )}{W_k(\rho_w^{(k)}  ) }.
	\]
	Let
	\[
	\Lambda_n = \left(X^{(1)}, \ldots, X^{(I-1)}, n - \sum_{k=1}^\ell X^{(k)}, X^{(I)} \ldots, X^{(\ell)}\right).
	\]
	Then
	\begin{align}
		\label{eq:tvapproxproduct}
		\lim_{n \to \infty} \sup_{A \subset \ndN^\ell} \left| \Prb{(P_1, \ldots, P_{\ell}) \in A} - \Prb{\Lambda_n \in A}  \right| = 0.
	\end{align}
\end{lemma}

The proof of Lemma~\ref{te:extended} uses the same approach as in~\cite{zbMATH06964621}.  We make the statement explicit since it is very useful in this context: In all settings considered in the present work  the partition function $(u_n)_{n \ge 0}$  satisfies the subexponentiality condition. This is detailed in the proof of each theorem, since the partition function $u_n$ is a constant multiple of the local probability $\Pr{S_N = n}$ of  a  randomly stopped sum $S_N$. Hence the asymptotic shape of product structures established in Lemma~\ref{te:extended} enables us to describe the asymptotic shape of so-called extended composition schemes without the need of case distinctions:

\begin{corollary}
	\label{co:extended}
	Given a power series $H(z) = \sum_{n \ge 0} h_n z^n$ with non-negative coefficients we might be interested in the random number $\tilde{N}_n$ and sizes $\tilde{K}_1, \ldots, \tilde{K}_{\tilde{N}_n}$ of $W$-components in the composition scheme $H(z)V(W(z))$, described by the probability generating function
	\[
	\Exb{x^{\tilde{N}_n} \prod_{i \ge 1} z_i^{\tilde{K}_i}} = \frac{[z^n] H(z) \sum_{\ell \ge 0} v_\ell \left(x \sum_{k \ge 0} w_k z_k z^k\right)^\ell }{[z^n]H(z)V(W(z))}.
	\]
	With $U(z) = V(W(z))$ as before, suppose that $V(z)$ and $W(z)$ satisfy the conditions of Theorem~\ref{te:dense} or  Theorem~\ref{te:mixture} or  Theorem~\ref{te:dilute}. Then, by Lemma~\ref{te:extended}, we have:
	\begin{enumerate}
		\item If $\lim_{n \to \infty} h_n / u_n = 0$, then the statements of Theorem~\ref{te:dense} / Theorem~\ref{te:mixture} / Theorem~\ref{te:dilute} / Corollary~\ref{co:poiz} / Corollary~\ref{co:pointprocess} for $N_n, K_1, \ldots, K_{N_n}$ hold analogously for $\tilde{N}_n, \tilde{K}_1, \ldots, \tilde{K}_{\tilde{N}_n}$. Here  Equation~\eqref{eq:convergentrephrase}  needs to be used instead of~\eqref{eq:tvapprox} to describe the asymptotic behaviour in  the convergent phase, since $\tilde{K}_1, \ldots, \tilde{K}_{\tilde{N}_n}$ need not sum up to $n$. Analogously for~\eqref{eq:mixturetvapprox}.
		\item  If $\lim_{n \to \infty} h_n / u_n = \infty$ and $(h_n)_{n \ge 0}$ satisfies the subexponentiality condition,
		then 
		\[
		(\tilde{N}_n, \tilde{K}_1, \ldots, \tilde{K}_{\tilde{N}_n}) \convdis (N', K_1', \ldots, K'_{N'})
		\]
		for a finite limit with  probability generating function
		\[
		\Exb{x^{N'} \prod_{i \ge 1} z_i^{K_i'}} = \frac{ \sum_{\ell \ge 0} v_\ell \left(x \sum_{k \ge 0} w_k z_k \rho_h^k\right)^\ell }{V(W(\rho_h))}.
		\]
		\item If $q := \lim_{n \to \infty} h_n / u_n \in ]0, \infty[$ exists, then the limiting behaviour is a mixture of the two previous cases, with the asymptotic probability for the first case  given by $q/(1+q)$, and for the second case by $1 / (1+q)$.
	\end{enumerate}  
\end{corollary}

By the same approach, we may deduce multi-variate limit laws for composition schemes of the form $\prod_{k=1}^\ell V_k(W_k(z))$ for $\ell$ constant.  We omit the straight-forward details.

\section{Proofs of main results}

\label{sec:proofs}

\subsection{The dense case}

Kolchin's representation theorem~\cite{MR865130,zbMATH02214057} builds a connection from Gibbs partitions to randomly stopped random walks in the case $\rho_u >0$.

\begin{proposition}[{ \cite[Thm. 1.2]{MR2245368}}]
	\label{pro:kolchin}
	Let $r >0$ denote a real number such that $0<U(r) < \infty$. Let $X$ and $N$ be a random non-negative integers with probability generating functions
	\begin{align}
		\label{eq:xn}
		\Ex{z^X} = \frac{W(rz)}{W(r)}, \qquad \text{and} \qquad 
		\Ex{z^N} = \frac{V(W(r)z)}{U(r)}.
	\end{align}
	Let $X_1, X_2, \ldots$ denote independent copies of $X$, that are also independent from $N$. Then for each integer $n \ge 1$ with $u_n>0$
	\[
	(K_1, \ldots, K_{N_n}) \eqdist \left( (X_1, \ldots, X_N) \,\,\Big\vert\,\, \sum_{i=1}^N X_i = n\right).
	\]
\end{proposition}

With this connection at hand, we are ready to prove the first main theorem.

\begin{proof}[Proof of Theorem~\ref{te:dense}]	
	The assumptions on the coefficients allow us to apply Proposition~\ref{pro:kolchin} for $r = \rho_u$, yielding 
	\begin{align}
		\label{eq:proko}
		(K_1, \ldots, K_{N_n}) \eqdist \left((X_1, \ldots, X_N) \,\,\Big\vert\,\, S_N = n\right).
	\end{align}
	Here $N, X$ denote independent random variables defined in  Equation~\eqref{eq:xn} for our choice  $r = \rho_u$, $(X_i)_{i \ge 1}$ denote independent copies of $X$, and
	$
	S_n = \sum_{i=1}^n X_i
	$
	for all integers $n \ge 0$. 
	
	Clearly the local limit theorem in Equation~\eqref{eq:densellt} implies the central limit theorem in Equation~\eqref{eq:denseclt}. Hence we are going to study the local probabilities.  Let $g(n)$ and $L(n)$ be defined as in Equations~\eqref{eq:defofg} and ~\eqref{eq:defofl1}. Let $\gamma$ be defined as in Equation~\eqref{eq:defofgamma}. Let $f$ denote the density function of the stable distribution $S_\alpha(\gamma, 1, 0)$.  Gnedenko's local limit theorem~\cite[Thm. 4.2.1]{MR0322926} in the lattice case states that
	\begin{align}
		\label{eq:tau}
		\tau_n = \sup_{k \in \ndZ} \left|g(n) n^{1/\alpha} \Pr{S_n = k} - f\left( \frac{k - n\mu}{g(n) n^{1/\alpha}} \right) \right|
	\end{align}
	with $\mu = \Ex{X}$ satisfies
	\begin{align}
		\label{eq:locallimit}
		\lim_{n \to \infty} \tau_n = 0.
	\end{align}
	By~\cite[Thm. 1, (ii), (iii)]{zbMATH07179510} and~\cite[Thm. 1.1]{zbMATH00902764} it holds that
	\begin{align}
		\label{eq:stoppedsum}
		\Pr{S_N = n} \sim \mu^{-1} \Pr{N= \lfloor n/\mu \rfloor }
	\end{align}
	as $n \to \infty$. 
	
	Let $\Omega \subset ]0,\infty[$ denote a compact subset.
	Using~\eqref{eq:locallimit},~\eqref{eq:stoppedsum}, and our assumptions on the coefficients of $V(z)$ it follows that  uniformly for all integers $\ell$ with \begin{align}
		\label{eq:rangeofl}
		x := \frac{\ell -  n/\mu}{L(n) n^{1/\alpha}} \in \Omega
	\end{align}
	it holds that
	\begin{align}
		\label{eq:t11simplifyme}
		\Pr{N_n = \ell} &= \Pr{N= \ell \mid S_N = n} \\
		&= \frac{ \Pr{N = \ell } \Pr{S_\ell = n}}{\Pr{S_N = n }} \nonumber \\
		&= \frac{ \Pr{N = \ell}}{\Pr{N= \lfloor n/\mu \rfloor }}\frac{1}{\mu^{-1} g(\ell) \ell^{1/\alpha}} \left( f\left( \frac{n - \ell \mu}{g(\ell) \ell^{1/\alpha}} \right) + o(1) \right). \nonumber
	\end{align}
	By standard properties of regularly varying functions~\cite[Thm. 1.2.4]{mikosch1999regular}, it follows that
	\begin{align}
		\label{eq:tictic}
		\frac{ \Pr{N = \ell}}{\Pr{N= \lfloor n/\mu \rfloor }} \sim 1,
	\end{align}
	and
	\begin{align}
		\label{eq:toctoc}
		\frac{1}{\mu^{-1} g(\ell) \ell^{1/\alpha}} \sim \frac{1}{L(n) n^{1/\alpha}},
	\end{align}
	and
	\[
	\frac{n - \ell \mu}{g(\ell) \ell^{1/\alpha}} \sim -x
	\]
	as $n \to \infty$, uniformly for all integers $\ell$ satisfying~\eqref{eq:rangeofl}. The density function $f$ is uniformly continuous, bounded, and positive. Hence
	\[
	f\left( \frac{n - \ell \mu}{g(\ell) \ell^{1/\alpha}} \right) \sim f\left(-x \right),
	\]
	and Equation~\eqref{eq:t11simplifyme} simplifies to
	\[
	\Pr{N = \ell} = \frac{1}{L(n) n^{1/\alpha}}  \left(f(-x) + o(1) \right).
	\]
	Since $X_\alpha(\gamma, -1, 0) \eqdist - X_\alpha(\gamma, 1, 0)$, we have $f(-x) = h(x)$. 
	
	In order to verify~Equation~\eqref{eq:densellt} it remains to treat the case $\delta n \le  \ell $ with  $|x| \to \infty$ as $n \to \infty$. Since $f(x) \to 0$ in this case, we have to show that
	\begin{align}
		\label{eq:ulttoshow}
		L(n)n^{1/\alpha}\Pr{N_n = \ell} \to 0.
	\end{align}
	If $\ell = O(n)$, then~\eqref{eq:ulttoshow} follows readily from~\eqref{eq:t11simplifyme}, since~\eqref{eq:tictic} still holds by~\cite[Thm. 1.2.4]{mikosch1999regular}, and although~\eqref{eq:toctoc} needs not hold we still have
	\[
	\frac{L(n) n^{1/\alpha}}{\mu^{-1} g(\ell) \ell^{1/\alpha}} = O(1).
	\]
	
	It remains to treat the case  $\ell= t_n n$ for some sequence $t_n$ satisfying $t_n \to \infty$. By the Potter bounds it follows that for each $\epsilon>0$ we have 
	\begin{align*}
		\frac{ \Pr{N = \ell}}{\Pr{N= \lfloor n/\mu \rfloor }} = O(t_n^\epsilon).
	\end{align*}
	Moreover, using again the Potter bounds and Equation~\eqref{eq:defofl1} we have 
	\begin{align*}
		\frac{L(n) n^{1/\alpha}}{\mu^{-1} g(\ell) \ell^{1/\alpha}} = O(t_n^{\epsilon - 1/\alpha}).
	\end{align*}
	Hence, taking $\epsilon < 2/\alpha$,  Equation~\eqref{eq:ulttoshow} follows readily from~\eqref{eq:t11simplifyme}. This completes the proof.
\end{proof} 

Throughout the rest of this section we continue using the notation from the proof of Theorem~\ref{te:dense}. With $\tau_n$ as in Equation~\eqref{eq:tau}, we set
\[
\tau_n^* = \sup_{k \ge n} \tau_k.
\]

We are going to make use of the following result by~\cite{zbMATH07179510}.
\begin{proposition}[{\cite[Lem. 2]{zbMATH07179510}}]
	\label{pro:blole2}
	Under the assumptions of Theorem~\ref{te:dense}, we have
	\[
	\lim_{n \to \infty} \sum_{\ell \,:\, |\ell \mu - n| \le t_n} \Pr{S_\ell = n} = \mu^{-1}
	\]
	for any sequence $(t_n)_{n \ge 1}$ satisfying
	\begin{align}
		\label{eq:bloeq}
		\frac{t_n}{g(n)n^{1/\alpha}} \to \infty, \qquad \frac{t_n^3}{n (g(n)n^{1/\alpha})^2} \to 0, \qquad \frac{t_n \tau^*_{\lfloor n / (2\mu) \rfloor} }{g(n)n^{1/\alpha}} \to 0.
	\end{align}
\end{proposition}
To be precise,~\cite[Lem. 2]{zbMATH07179510} was formulated for Assumption i) on the weight-sequences. But its proof also applies to the case in Assumption ii) without the need of any modification.

\begin{proof}[Proof of Lemma~\ref{le:independent}]
	Let $(s_n)_{n \ge 1}$ denote a sequence of positive integers satisfying
	\begin{align}
		\label{eq:sreq}
		\frac{s_n}{L(n)n^{1/\alpha}} \to \infty \qquad \text{and} \qquad s_n \le n/(2 \mu).
	\end{align}
	Let $(t_n)_{n \ge 1}$ denote a sequence of positive real numbers such that 
	\begin{align}
		\label{eq:bw1}
		\frac{t_n}{L(n)n^{1/\alpha}} \to \infty \qquad \text{and} \qquad t_n = o(s_n).
	\end{align}
	We set
	\[
	\delta_n := \lfloor n/ \mu \rfloor - s_n.
	\]
	Let $\epsilon>0$. By the local limit theorem~\eqref{eq:locallimit} we may choose a constant $M>0$ sufficiently large so that
	\begin{align}
		\label{eq:bwair}
		\Pr{|S_{\delta_n} - \delta_n \mu| \le M g(n) n^{1/\alpha}} > 1 - \epsilon
	\end{align}
	for all $n$. Let \[
	k_1, \ldots, k_{\delta_n} \in \{k \ge 0 \mid \Pr{X=k} > 0\}
	\] denote integers with the property  that the sum
	\[
	y := \sum_{i=1}^{\delta_n} k_i 
	\]
	satisfies
	\begin{align}
		\label{eq:bw2}
		|y - \delta_n \mu| \le M g(n) n^{1/\alpha}.
	\end{align}
	Then, by Equations~\eqref{eq:proko} and~\eqref{eq:stoppedsum}, and standard properties of regularly varying functions~\cite[Thm. 1.2.4]{mikosch1999regular} it follows that
	\begin{align*}
		&\frac{\Prb{(K_1, \ldots, K_{ \delta_n})= (k_1, \ldots, k_{\delta_n}), |N_n - n/\mu| \le t_n} }{\Prb{(X_1, \ldots, X_{\delta_n})= (k_1, \ldots, k_{\delta_n})  }} \\
		&\qquad = \frac{\Prb{ |N - n/\mu| \le t_n, S_{N- \delta_n} = n - y }}{\Prb{S_N = n}} \\
		&\qquad \sim \mu\sum_{\ell \,:\, |\ell - n/\mu| \le t_n} \frac{\Pr{N=\ell}}{\Pr{N= \lfloor n / \mu \rfloor}}  \Pr{S_{\ell - \delta_n} = n -y}\\
		&\qquad \sim \mu \sum_{\ell \,:\, |\ell - n/\mu| \le t_n}\Pr{S_{\ell - \delta_n} = n -y}.
	\end{align*}
	Note that here we have used our assumption $t_n = o(s_n)$, which ensures that $\delta_n < n/\mu - t_n$ for large enough $n$. 
	
	In order to apply Proposition~\ref{pro:blole2}, we set
	\begin{align*}
		\tilde{n} &:= n-y = s_n \mu + O(g(n)n^{1/\alpha})  \\
		\tilde{\ell} &:= \ell - \delta_n.
	\end{align*}
	Hence 
	\begin{align*}
		\tilde{\ell} \mu - \tilde{n} = \mu(\ell - n/\mu) + O(g(n)n^{1/\alpha}).
	\end{align*}
	Thus, summing over all integers $\ell$ with $|\ell - n/\mu| \le t_n$ means that $\tilde{\ell}$ ranges over all integers satisfying a specific inequality of the form 
	\[
	-\mu t_n(1+o(1)) \le \tilde{\ell} \mu - \tilde{n} \le \mu t_n (1+o(1)).
	\]
	Thus, in order to apply Proposition~\ref{pro:blole2} we need to check the conditions of~\eqref{eq:bloeq} for this case, which simplify to
	\begin{align*}
		\frac{t_n}{g(\tilde{n})\tilde{n}^{1/\alpha}} \to \infty, \qquad \frac{t_n^3}{\tilde{n} (g(\tilde{n})\tilde{n}^{1/\alpha})^2} \to 0, \qquad \frac{t_n \tau^*_{\lfloor \tilde{n} / (2\mu) \rfloor} }{g(\tilde{n})\tilde{n}^{1/\alpha}} \to 0.
	\end{align*}
	Note that $\tilde{n} \sim \mu s_n$ and hence $g(\tilde{n})\tilde{n}^{1/\alpha} \sim g(s_n) s_n^{1/\alpha} \mu^{1/\alpha}$. Hence, writing 
	\[
	t_n = x_n g(s_n) s_n^{1/\alpha}
	\]
	for $x_n >0$, these conditions are  equivalent to
	\begin{align}
		\label{eq:bw3}
		x_n \to \infty,\qquad x_n^3 g(s_n) s_n^{1/\alpha -1} \to 0, \qquad x_n \tau^*_{ s_n/2(1+o(1))} \to 0.
	\end{align}
	Equation~\eqref{eq:bw1} may be rephrased by
	\begin{align}
		\label{eq:bw4}
		x_n \frac{g(s_n) s_n^{1/\alpha}}{g(n)n^{1/\alpha}} \to \infty \qquad \text{and} \qquad x_n s_n^{1/\alpha-1} g(s_n) = o(1).
	\end{align}
	
	The question now is, how slowly can we allow $s_n$  to tend to infinity (subject to~\eqref{eq:sreq}) so that there still exists a sequence $(x_n)_{n \ge 1}$ that satisfies both~\eqref{eq:bw3} and~\eqref{eq:bw4}. Note that the first two limits of~\eqref{eq:bw3} already imply the second limit of~\eqref{eq:bw4}. The second limit of~\eqref{eq:bw4} also implies $x_n \to \infty$, hence we may summarize the conditions on $(x_n)_{n \ge 1}$ by
	\begin{align}
		\label{eq:summarized}
		x_n \frac{g(s_n) s_n^{1/\alpha}}{g(n)n^{1/\alpha}} \to \infty,\qquad x_n \left(  g(s_n) s_n^{1/\alpha -1} \right)^{1/3} \to 0, \qquad x_n \tau^*_{ s_n/2(1+o(1))} \to 0.
	\end{align}
	The difficulty is obtaining a concrete bound for $\tau_n^*$. In general, we only know that $\tau_n^* = o(1)$, ensuring that for $\delta'>0$ sufficiently small we may set $s_n = \lfloor \delta' n\rfloor $ and find a sequence $x_n$ that tends to infinity sufficiently slowly so that~\eqref{eq:summarized} is satisfied. If $\Ex{X^3} < \infty$ then $\tau_n^* = O(1/\sqrt{n})$ by~\cite[Thm. 6, p. 197]{zbMATH03504209}, allowing us to take $s_n = o(n)$ with $s_n / n^{3/4} \to \infty$ and set $x_n = n^{1/8}$. If $\Ex{X^2}<\infty$ and $\Ex{X^2 \one_{X \ge n}} = O(n^{-r})$ for $0<r<1$, then $\tau_n^* = O(1/n^{r/2})$ by~\cite[Thm. 4.5.3]{MR0322926}, allowing us to take $x_n = n^{1/8}$  for $s_n / n^{3/4} \to \infty$ if $r \ge 1/3$, and $x_n = n^{\frac{r}{2(1+r)}}$ for $r<1/3$. 
	
	For these choices of sequences, Proposition~\ref{pro:blole2} yields
	\[
	\lim_{n \to \infty}	\sum_{\ell \,:\, |\ell - n/\mu| \le t_n}\Pr{S_{\ell - \delta_n} = n -y} = \mu^{-1}.
	\]
	Hence
	\begin{align}
		\label{eq:noooope}
		\frac{\Prb{(K_1, \ldots, K_{ \delta_n})= (k_1, \ldots, k_{\delta_n}), |N_n - n/\mu| \le t_n} }{\Prb{(X_1, \ldots, X_{\delta_n})= (k_1, \ldots, k_{\delta_n})  }} \to 1
	\end{align}
	as $n \to \infty$, 
	uniformly for all $(k_1, \ldots, k_{\delta_n})$ satisfying~\eqref{eq:bw2}. Furthermore, by Theorem~\ref{te:dense} we have $\Pr{|N_n - n/\mu| \le t_n} \to 1$ as $n$ tends to infinity. The limit~\eqref{eq:noooope} and Inequality~\eqref{eq:bwair} entail that for all large enough $n$
	\[
	\Prb{ \left|\sum_{i=1}^{\delta_n} K_i  - \delta_n \mu \right|  \le M g(n) n^{1/\alpha}} > 1 - 2 \epsilon.
	\]
	As $\epsilon>0$ was arbitrary, it follows that 
	\begin{align}
		\frac{\Prb{(K_1, \ldots, K_{ \min(\delta_n, N_n)})= (k_1, \ldots, k_{\delta_n})} }{\Prb{(X_1, \ldots, X_{\delta_n})= (k_1, \ldots, k_{\delta_n})  }} \to 1
	\end{align}
	as $n \to \infty$, uniformly for all $(k_1, \ldots, k_{\delta_n}) \in \cG_n$ for some set $\cG_n$ satisfying 
	\[
	\Prb{(K_1, \ldots, K_{ \min(\delta_n, N_n)}) \in \cG_n } > 1 - \epsilon 
	\]
	and
	\[
	\Prb{(X_1, \ldots, X_{\delta_n}) \in \cG_n  } > 1 - \epsilon
	\]
	for all large enough $n$. Using again that $\epsilon>0$ was arbitrary, it follows that
	\begin{align}
		\label{eq:asdfqwerasdf}
		\lim_{n \to \infty}	d_{\mathrm{TV}}\left( (K_1, \ldots, K_{\min(N_n, \delta_n)}), (X_1, \ldots, X_{\delta_n} ) \right)  = 0.
	\end{align} 
	In particular, since we may always take $s_n = \delta'n$ for $\delta'>0$ sufficiently small, it follows that~\eqref{eq:asdfqwerasdf} still holds for at least one sequence $(s_n)_{n \ge 1}$ satisfying $s_n = o(n)$.
\end{proof}

\begin{proof}[Proof of Lemma~\ref{le:contiguity}]	
	For all $M_1, M_2>0$ and $\delta>0$ we define the collection $\cF_{n, \delta, M_1,M_2}$ of finite sequences \[
	F = (\ell, k_1, \ldots, k_{\ell - \lfloor \delta n \rfloor })
	\]  with  integers $\ell, k_1, k_2, \ldots \ge 0$ satisfying the following properties:
	\begin{enumerate}[\qquad a)]
		\item $\Pr{N = \ell}>0$  and $\Pr{X=k_i} >0$ for all  $1 \le i \le \ell - \lfloor \delta n \rfloor $,
		\item $|\sum_{i=1}^{\ell - \lfloor \delta n \rfloor} k_i - (\ell - \delta n) \mu | \le  M_1 g(n)n^{1/\alpha}$,
		\item $|\ell\mu  - n| \le   M_2 L(n) n^{1/\alpha}$.
	\end{enumerate}
	For any such sequence $F \in \cF_{n, \delta, M_1,M_2}$ we set $y = \sum_{i=1}^{\ell - \lfloor \delta n \rfloor } k_i$.  By Equation~\eqref{eq:proko} and Assumption a), it follows that
	\begin{align}
		\label{eq:quotient}
		&\frac{\Prb{ (N_n, (K_i)_{1\le i \le  N_n - \lfloor \delta n \rfloor })  = F } } {\Prb{ (N_n, (X_i)_{1\le i \le  N_n - \lfloor \delta n \rfloor })  = F  }} \nonumber \\
		&\qquad = \frac{\Prb{ (N, (X_i)_{1\le i \le  N - \lfloor \delta n \rfloor })  = F  \mid S_N = n}}{ \Prb{N = \ell \mid S_N = n} \prod_{i=1}^{\ell - \lfloor \delta n \rfloor} \Pr{X = k_i} } \nonumber \\	
		&\qquad = \frac{\Prb{ (N, (X_i)_{1\le i \le  N - \lfloor \delta n \rfloor })  = F  , \sum_{i = \ell - \lfloor \delta n \rfloor +1}^\ell X_i = n-y}}{ \Prb{N = \ell , S_\ell = n} \prod_{i=1}^{\ell - \lfloor \delta n \rfloor} \Pr{X = k_i} } \nonumber \\	
		&\qquad = \frac{\Prb{S_{\lfloor \delta n \rfloor} = n-y}}{ \Prb{S_\ell = n} } \nonumber 		
	\end{align}
	Using~\eqref{eq:locallimit} it follows that
	\begin{align*}
		\frac{\Prb{S_{\lfloor \delta n \rfloor} = n-y}}{ \Prb{S_\ell = n} } = \frac{ g(\ell) \ell^{1/ \alpha} }{g(\lfloor \delta n \rfloor)\lfloor \delta n \rfloor^{1/\alpha} } \frac{f\left( \frac{n - y - \lfloor \delta n \rfloor \mu}{ g(\lfloor \delta n \rfloor) \lfloor \delta n \rfloor^{1/\alpha} }  \right) + o(1)}{f\left( \frac{n - \ell \mu}{ g(\ell) \ell^{1/\alpha}} \right) + o(1)}.
	\end{align*}
	By standard properties of regularly varying functions~\cite[Thm. 1.2.4]{mikosch1999regular}, Equation~\eqref{eq:defofl1}, Assumptions b) and c), and the fact that the density function $f$ is bounded, positive, and uniformly continuous it follows that this expression is bounded from above and below by constants that depend on $M_1$ and $M_2$, but not on $n$.
	
	Furthermore, for any $\epsilon>0$ we may take $M_1$ and $M_2$ sufficiently large so that by~\eqref{eq:locallimit}, Theorem~\ref{te:dense}, and Lemma~\ref{le:independent} we have
	\[
	\Pr{(N_n, (K_i)_{1\le i \le  N_n - \lfloor \delta n \rfloor }) \in  \cF_{n, \delta, M_1,M_2} } > 1 - \epsilon
	\]
	and
	\[
	\Pr{(N_n, (X_i)_{1\le i \le  N_n - \lfloor \delta n \rfloor })  \in  \cF_{n, \delta, M_1,M_2} } > 1 - \epsilon.
	\]
	This completes the proof.
\end{proof}

\begin{proof}[Proof of Theorem~\ref{te:functional}]
	The sums $S_n = \sum_{i=1}^n X_i$ satisfy the functional limit theorem
	\begin{align}
		\left( \frac{S_{\lfloor n s \rfloor} - ns \mu}{g(n)n^{1/\alpha}} \right) \convdis (Y_s,\,\, 0 \le s \le 1)
	\end{align}	
	as $n \to \infty$. By Lemma~\ref{le:independent} it follows that for each $0< \delta < \mu^{-1}$
	\[
	\left( \frac{\sum_{i=1}^{\min(N_n, \lfloor s n \rfloor)}K_i - n\mu s  }{g(n)n^{1/\alpha}}, \,\, 0 \le s \le \delta \right) \convdis ( Y_s, \,\, 0 \le s \le \delta).
	\]
	Since $(Y_s)_{s \ge 0}$ is almost surely continuous at $\delta$, it follows by
	Equation~\eqref{eq:defofl1}, Theorem~\ref{te:dense}, and~\cite[Lem. 5.7]{MR2946438} that
	\[
	\left( \frac{\sum_{i=1}^{N_n s}K_i - \mu N_n s  }{L(n)n^{1/\alpha}}, \,\, 0 \le s \le \delta \right) \convdis (\mu Y_s, \,\, 0 \le s \le \delta).
	\]
	By a time-reversal argument, it follows that
	\[
	\left( \frac{\sum_{i=1}^{\lfloor s N_n \rfloor}K_i - N_n s \mu }{L(n)n^{1/\alpha}}, \,\, 0 \le s \le 1 \right) \convdis (\mu Y_s, \,\, 0 \le s \le 1).
	\]
\end{proof}

\begin{proof}[Proof of Corollary~\ref{co:maxcomp}]
	Note that our assumptions entail that $X$ is not bounded. In any case, the statement~\eqref{eq:megatv} would be trivial when $X$ is bounded, because in this case Lemma~\ref{le:independent} ensures that both sides of the equation are with high probability equal to tuples with constant coordinates  equal to the largest integer that $X$ attains with positive probability. 
	
	Let $(s_n)_{n \ge 1}$ be as in Lemma~\ref{le:independent}. We are going to argue that for suitable choices of sequences $(t_n)_{n \ge 1}$  of positive integers with $t_n \le s_n$ the following  two statements hold:
	\begin{enumerate}[\qquad a)]
		\item 	If $\tilde{K}_{(1)} \ge \tilde{K}_{(2)} \ge \ldots$ denote the initial segment $(K_1, \ldots, K_{\min(\lfloor n/\mu \rfloor - s_n, N_n)})$ arranged in decreasing order, then
		\begin{align}
			\label{eq:claimaa}
			\lim_{n \to \infty} \Pr{ \tilde{K}_{(t_n)} <  \max\{K_i \mid {\lfloor n/\mu \rfloor - s_n \le i \le N_n} \} } = 0
		\end{align}
		\item If $\tilde{X}_{(1)} \ge \tilde{X}_{(2)} \ge \ldots$ denote $\lfloor n/\mu\rfloor - s_n $ independent copies of $X$ arranged in decreasing order, then
		\begin{align}
			\label{eq:claimbb}
			\lim_{n \to \infty} d_{\mathrm{TV}}\left( (\tilde{X}_{(1)}, \ldots, \tilde{X}_{(t_n)})  , (X_{(1)}, \ldots, X_{(t_n)})    \right) = 0.
		\end{align}
	\end{enumerate}
	
	These statements suffice to prove Equation~\eqref{eq:megatv}. Indeed, for any set~$A$ it follows by applying first Equation~\eqref{eq:claimbb}, then Lemma~\ref{le:independent}, and then Equation~\eqref{eq:claimaa} that
	\begin{align*}
		\Pr{ (X_{(1)}, \ldots, X_{(t_n)}) \in A} &= o(1) +  \Pr{ (\tilde{X}_{(1)}, \ldots, \tilde{X}_{(t_n)}) \in A} \\
		&= o(1) +  \Pr{ (\tilde{K}_{(1)}, \ldots, \tilde{K}_{(t_n)}) \in A} \\
		&= o(1) +  \Pr{ ({K}_{(1)}, \ldots, {K}_{(t_n)}) \in A}
	\end{align*}
	with uniform  $o(1)$ terms that do not depend on $A$. This implies Equation~\eqref{eq:megatv}.
	
	It remains to verify Equations~\eqref{eq:claimaa} and~\eqref{eq:claimbb}. Our assumptions on $X$ entail (in all cases) that there exists $A>0$ such that for all $\epsilon>0$ there exists a sequence  $\kappa(n)$ with
	\begin{align}
		\label{eq:epsilonlimit}
		\epsilon/ A \le  n \Pr{X \ge \kappa(n)} \le A \epsilon
	\end{align}
	for large enough $n$.
	By Theorem~\ref{te:dense} we know that with high probability $N_n - 2s_n < \lfloor n/\mu \rfloor - s_n$.  Using Lemma~\ref{le:independent} and a time-reversal argument it follows that
	\begin{align}
		\label{eq:matogege}
		&\Prb{ \max\{K_i \mid {\lfloor n/\mu \rfloor - s_n \le i \le N_n} \} \ge \kappa(s_n) } \\
		&\qquad \le o(1) + \Prb{ \max(X_1, \ldots, X_{2s_n}) \ge \kappa(s_n)} \nonumber \\
		&\qquad \le o(1) + 2s_n \Prb{ X \ge \kappa(s_n)} \nonumber \\
		&\qquad \le o(1) + 2 A \epsilon.\nonumber 
	\end{align}
	Furthermore, by Lemma~\ref{le:independent}, Equation~\eqref{eq:epsilonlimit}, and $s_n = o(n)$ and  $t_n = o(n)$ 
	\begin{align}
		\label{eq:matogege2}
		&\Prb{ \tilde{K}_{(t_n)} <\kappa(s_n) } \\
		&\qquad = o(1)+ \Prb{ \tilde{X}_{(t_n)} < \kappa(s_n) } \nonumber \\
		&\qquad = o(1) + \sum_{k=0}^{t_n-1} \binom{\lfloor n/\mu \rfloor - s_n}{k} \Pr{X \ge \kappa(s_n)}^{k} \Pr{X < \kappa(s_n)}^{\lfloor n/\mu \rfloor - s_n - k} \nonumber\\
		&\qquad \le o(1) + \Pr{X < \kappa(s_n)}^{ n/\mu (1+o(1))} \sum_{k=0}^{t_n-1}  \left(\frac{n}{\mu s_n} s_n\Pr{X \ge \kappa(s_n)}\right)^{k} \nonumber\\
		&\qquad \le o(1) + \exp\left( -\Theta(1) \frac{n\epsilon}{\mu s_n}  \right) t_n \left(\frac{n\epsilon}{\mu s_n} \Theta(1)\right)^{t_n} \nonumber\\
		&\qquad \le o(1) + \exp\left( -\Theta(1) \frac{n\epsilon}{\mu s_n} + \log(t_n) + t_n\log\left(\frac{n \epsilon}{\mu s_n}\Theta(1)\right) \right).\nonumber
	\end{align}
	
	As $\epsilon>0$ was arbitrary, \eqref{eq:matogege} and~\eqref{eq:matogege2}  verify Equation~\eqref{eq:claimaa} when $t_n = O(1)$. Consequently,~\eqref{eq:claimaa} also holds when $t_n$ tends to infinity sufficiently slowly. In the cases with additional assumptions on the moments on $X$, Lemma~\ref{le:independent} ensures that there exists $0<r\le1/3$ (that corresponds on the case under consideration) such that for any sequence $a_n \to \infty$ we may set $s_n = a_n n^{\frac{1}{1+r}}$. In particular, if $t_n = o(\frac{n^{\frac{r}{1+r}}}{\log n})$, we choose $a_n$ to tend to infinity sufficiently slowly so that for each $\epsilon<0$ the bound in~\eqref{eq:matogege2} tends to zero. Together with \eqref{eq:matogege} this verifies Equation~\eqref{eq:claimaa}.
	
	Since we verified Equation~\eqref{eq:claimaa} by reducing everything to the i.i.d. model, Equation~\eqref{eq:claimbb} follows by the same arguments. Hence the proof is complete.
\end{proof}

\subsection{The convergent case}

\begin{proof}[Proof of Theorem~\ref{te:convergent}]
	Equation~\eqref{eq:convcomponents} follows from the total variation approximation~\eqref{eq:tvapprox}. Furthermore, by standard properties of the total variation distance, Equation~\eqref{eq:tvapprox} is equivalent to Equation~\eqref{eq:convergentrephrase}. 
	
	Hence it suffices to verify \eqref{eq:convergentrephrase}. To this end, let $1 \le j \le m$ and $x_1, \ldots, x_{m-1} \ge 0$ be integers. Set 
	\[
	y = \sum_{i=1}^{m-1} x_i.
	\]	
	Applying Proposition~\ref{pro:kolchin} for $r = \rho_w$ yields for all $n$ with $n > y +  \max_{1 \le i \le m-1} x_i$
	\begin{align*}
		&\Pr{\phi(K_1, \ldots, K_{N_n}) = (x_1, \ldots, x_{j-1}, *, x_j, \ldots, x_{m-1})} \\
		&\qquad= \Pr{(K_1, \ldots, K_{N_n}) = (x_1, \ldots, x_{j-1}, n - y, x_j, \ldots, x_{m-1})} \\
		&\qquad= \frac{\Pr{ (X_1, \ldots, X_N) = (x_1, \ldots, x_{j-1}, n - y, x_j, \ldots, x_{m-1})}\Pr{N=m}}{\Pr{S_N = n}} \\
		&\qquad= \frac{\Pr{X= n-y} \Pr{N=m}}{\Pr{S_N = n} } \prod_{i=1}^{m-1} \Pr{X= x_i}.
	\end{align*}
	Each of the cases considered in Theorem~\ref{te:convergent} implies that
	\begin{align}
		\Pr{X= n-y} \sim \Pr{X=n}
	\end{align}
	as $n \to \infty$. Moreover, each of the cases implies
	\begin{align}
		\Pr{S_N= n} \sim \Ex{N}\Pr{X=n}.
	\end{align}
	Indeed, this follows from~\cite[Thm. 1, (i), (iii)]{zbMATH07179510} in the case i),~\cite[Thm. 4.31]{MR3097424} in the case ii), and~\cite[Theorems 2--5]{zbMATH07179510} in the remaining cases. Using these asymptotics and 
	\begin{align}
		\Pr{\hat{N} = m} = \frac{m \Pr{N=m}}{\Ex{N}},
	\end{align}
	it follows that
	\begin{align*}
		&\lim_{n \to \infty} \Pr{\phi(K_1, \ldots, K_{N_n}) = (x_1, \ldots, x_{j-1}, *, x_j, \ldots, x_{m-1})} \\
		&\qquad=  \frac{\Pr{N=m}}{\Ex{N} } \prod_{i=1}^{m-1} \Pr{X= x_i} \\
		&\qquad= \Prb{ \left(X_1, \ldots, X_{J-1}, *, X_{J}, \ldots, X_{\hat{N}-1}\right) = (x_1, \ldots, x_{j-1}, *, x_j, \ldots, x_{m-1})}.
	\end{align*}
	This completes the proof.
\end{proof}

\subsection{The mixture case}

\begin{proof}
	The assumptions on the weight sequences imply by~\cite[Thm. 1, (ii)]{zbMATH07179510} that
	\begin{align}
		\label{eq:stoppedsummixture}
		\Pr{S_N = n} \sim \Ex{N}\Pr{X=n} + \mu^{-1} \Pr{N= \lfloor n/\mu \rfloor }
	\end{align}
	as $n \to \infty$, and
	\begin{align}
		\lim_{n \to \infty} \frac{\mu^{-1} \Pr{N= \lfloor n/\mu \rfloor }}{\Ex{N}\Pr{X=n} } = p \in]0,1[.
	\end{align}
	Furthermore, by \cite[Eq. (131), (132)]{zbMATH07179510} we have
	\begin{align}
		\label{eq:yo0}
		\Pr{S_N = n, N \ge n / (2 \mu)} \sim \mu^{-1} \Pr{N= \lfloor n/\mu \rfloor },
	\end{align}
	and consequently
	\begin{align}
		\label{eq:yo1}
		\Pr{S_N = n, N < n / (2 \mu)} \sim\Ex{N}\Pr{X=n}.
	\end{align}
	We define the set 
	\[
	\mathfrak{E}_n = \left\{ (x_1, \ldots, x_k) \in \ndN_0^k \,\,\Big\vert\,\, k \ge n/(2 \mu), \sum_{i=1}^kx_i = n \right\}
	\]
	and the event 
	\[
	\cE_n= \{(K_1, \ldots, K_{N_n}) \in \mathfrak{E}_n\}.
	\]
	Using Proposition~\ref{pro:kolchin} for $r = \rho_w$, it follows that 
	\[
	\lim_{n \to \infty} \Pr{\cE_n} = p.
	\]

	Using Equation~\eqref{eq:yo0}, it follows that for sufficiently large integers $n$ it holds uniformly for all integers $\ell \ge \delta n$
	that
	\begin{align*}
		\Pr{N_n = \ell \mid \cE_n} &= \Pr{N= \ell \mid S_N = n, N \ge n / (2 \mu)} \\
		&= \frac{ \Pr{N = \ell } \Pr{S_\ell = n}}{\Pr{S_N = n ,N \ge n / (2 \mu)}} \nonumber \\
		&= \frac{ \Pr{N = \ell}}{\Pr{N= \lfloor n/\mu \rfloor }}\frac{1}{\mu^{-1} g(\ell) \ell^{1/\alpha}} \left( f\left( \frac{n - \ell \mu}{g(\ell) \ell^{1/\alpha}} \right) + o(1) \right), \nonumber
	\end{align*}
	with  $f$ denoting the density function of the stable distribution $S_\alpha(\gamma, 1, 0)$. By identical arguments as in the proof of Theorem~\ref{te:dense}, this simplifies to
	\begin{align*}
		\Pr{N_n = \ell \mid \cE_n} &= \frac{1}{L(n) n^{1/\alpha}}  \left(h(x) + o(1) \right).
	\end{align*}
	This verifies the local limit theorem~\eqref{eq:mixturellt}, which also implies the central limit theorem~\eqref{eq:mixtureclt}.
	
	Let $1 \le j \le m$ and $x_1, \ldots, x_{m-1} \ge 0$ be integers, and set
	\[
	y = \sum_{i=1}^{m-1} x_i.
	\]	
	Using Equation~\eqref{eq:yo1}, it follows that for all sufficiently large $n$ 
	\begin{align*}
		&\Pr{\phi(K_1, \ldots, K_{N_n}) = (x_1, \ldots, x_{j-1}, *, x_j, \ldots, x_{m-1}) \mid \cE_n^c} \\
		&\qquad= \Pr{(K_1, \ldots, K_{N_n}) = (x_1, \ldots, x_{j-1}, n - y, x_j, \ldots, x_{m-1}) \mid N < n / (2\mu)} \\
		&\qquad= \frac{\Pr{ (X_1, \ldots, X_N) = (x_1, \ldots, x_{j-1}, n - y, x_j, \ldots, x_{m-1})}\Pr{N=m}}{\Pr{S_N = n, N < n / (2\mu)}} \\
		&\qquad \sim  \frac{\Pr{X= n-y} \Pr{N=m}}{\Ex{N}\Pr{X=n} } \prod_{i=1}^{m-1} \Pr{X= x_i} \\
		&\qquad \sim  \Prb{ \left(X_1, \ldots, X_{J-1}, *, X_{J}, \ldots, X_{\hat{N}-1}\right) = (x_1, \ldots, x_{j-1}, *, x_j, \ldots, x_{m-1})}.
	\end{align*}
	This verifies
	\begin{align*}
		(\phi(K_1, \ldots, K_{N_n}) \mid \cE_n^c) \convdis \left(X_1, \ldots, X_{J-1}, *, X_{J}, \ldots, X_{\hat{N}-1}\right).
	\end{align*}
	Equations~\eqref{eq:mixturetvapprox} and~\eqref{eq:mixturecomponents} readily follow.
\end{proof}

\subsection{The dilute case}
\label{sec:proofdilute}

\begin{proof}[Proof of Theorem~\ref{te:dilute}]
	The local limit theorem in~\eqref{eq:dilutellt} implies the distributional convergence in~\eqref{eq:diluteclt}. Hence we will determine the local probabilities. Proposition~\ref{pro:kolchin} for $r = \rho_w$ entails for all $\ell \in \ndN$
	\begin{align}
		\label{eq:notnotquite}
		\Pr{N_n = \ell} &= \Pr{N= \ell \mid S_N = n} \\
		&= \frac{ \Pr{N = \ell } \Pr{S_\ell = n}}{\Pr{S_N = n }}. \nonumber
	\end{align}
	The assumptions on the weight sequences imply by~\cite[Thm. 1, (iv) and Eq. (133)]{zbMATH07179510} that
	\begin{align}
		\label{eq:stoppedsumdilute}
		\Pr{S_N = n} \sim n^{-1-\alpha(b-1)} V(W(\rho_w))^{-1}L_v(n^\alpha) \alpha \Ex{(X_{\alpha}(\gamma, 1,0))^{\alpha(b-1)}}
	\end{align}
	as $n \to \infty$. See~\cite[Thm. 7.4]{2011arXiv1112.0220J} for the choice of $\gamma$.
	
	We may write $\ell = n^{\alpha}x$ for some $x \in \ndR$. For now, let us consider the case $x \in \Omega$ for some fixed compact subset $\Omega \subset ]0,\infty[$. Using Gnedenko's local limit theorem~\cite[Thm. 4.2.1]{MR0322926}   and Equation~\eqref{eq:stoppedsumdilute} it follows that uniformly for all $\ell$ with $x \in \Omega$
	\begin{align}
		\label{eq:tosimplifydilute}
		\Pr{N_n = \ell} &= \frac{L_v(\ell)\ell^{-b} \ell^{-1/\alpha}(f(n / \ell^{1/\alpha}) + o(1))}{n^{-1-\alpha(b-1)} L_v(n^\alpha) \alpha \Ex{(X_{\alpha}(\gamma, 1,0))^{\alpha(b-1)}}}.
	\end{align}
	Since $\Omega \subset ]0, \infty[$ is compact, it follows by standard properties of regularly varying functions~\cite[Thm. 1.2.4]{mikosch1999regular} that
	\begin{align}
		\label{eq:yoqwer}
		\frac{L_v(\ell)}{L_v(n^{\alpha})} \to 1
	\end{align}
	uniformly for all $\ell$ with $x \in \Omega$. Hence Equation~\eqref{eq:tosimplifydilute} simplifies to
	\begin{align}
		\Pr{N_n = \ell} = \frac{x^{-b - 1/\alpha}(f(1/x^{1/\alpha}) + o(1))}{n^{\alpha}  \alpha \Ex{(X_{\alpha}(\gamma, 1,0))^{\alpha(b-1)}}}.
	\end{align}

	The density function $f$ is  uniformly continuous, hence the associated density $\tilde{f}$ satisfies $\lim_{x\to \infty} \tilde{f}(x) = 0$. Hence in order to verify~\eqref{eq:dilutellt} it remains to show that if $\ell_n  = n^{\alpha} x_n$ with  $x_n \to \infty$ then 
	\begin{align}
		\label{eq:eqremtoshow}
		n^{\alpha}\Pr{N_n = \ell_n} \to 0.
	\end{align} 
	To this end, note that by the Potter bounds it holds for any $\epsilon>0$ that
	\begin{align}
		\label{eq:instead}
		\frac{ L_v(x_n n^{\alpha})}{L_v(n^{\alpha})} = O(1)\max(x_n^\epsilon, x_n^{-\epsilon}).
	\end{align}
	Using~\eqref{eq:instead} and~\eqref{eq:tosimplifydilute} it follows that
	\begin{align}
		n^{\alpha}\Pr{N_n = \ell_n} \le \max(x_n^\epsilon, x_n^{-\epsilon}) x_n^{-b - 1/\alpha} o(1).
	\end{align}
	If we choose $\epsilon$ sufficiently small, then this bound tends to zero. 
\end{proof}

\begin{proof}[Proof of Corollary~\ref{co:poiz}]
	Let $\epsilon>0$. By Theorem~\ref{te:dilute} we may select a sufficiently large compact interval $\Omega \subset ]0, \infty[$ such that \[
	\Pr{N_n/n^{\alpha} \in \Omega} > 1 - \epsilon
	\] for all sufficiently large $n$. By a slight abuse of notation, we let $n^\alpha \Omega$ denote the collection of all integers $\ell$ with $\ell / n^\alpha \in \Omega$. Hence, for all integers $k \ge 0$
	\begin{align}
		\label{eq:analogouspriel}
		\Pr{\#_{k_n} P_n = k} &= R + \sum_{\ell \in n^\alpha \Omega} \Pr{N_n = \ell} \Pr{\#_{k_n} P_n = k \mid N_n = \ell} \\ 
		&= R + \sum_{\ell \in n^\alpha \Omega} \Pr{N_n = \ell}  \binom{\ell}{k} \Pr{X=k_n}^k \frac{\Pr{S_{\ell-k} = n - k k_n, X_i \ne k_n \text{ for all $1 \le i \le \ell-k$}  }}{\Pr{S_\ell = n}}, \nonumber
	\end{align}
	with $R := \Pr{\#_{k_n} P_n = k, N_n \notin n^\alpha \Omega} < \epsilon$. Let $\bar{X}_1, \bar{X}_2, \ldots$ denote independent copies of $(X \mid X \ne k_n)$ and set $\bar{S}_n = \sum_{i=1}^n \bar{X}_i$.  If $k_n \to \infty$  then we have uniformly for   $\ell \in n^{\alpha} \Omega$ and all $ 0 \le k \le o(n^{\alpha})$ (with $o(n^{\alpha})$ denoting  $n^\alpha$ times some sequence that we may choose to tend to zero arbitrarily slowly)
	\begin{align*}
		\Pr{S_{\ell-k} = n - k k_n, X_i \ne k_n \text{ for all $1 \le i \le \ell-k$}} &={\Pr{\bar{S}_{\ell-k} = n - k k_n}}{(1- \Pr{X=k_n})^{\ell-k}} \\
		&\sim \frac{\Pr{\bar{S}_{\ell-k} = n - k k_n}}{\exp(\ell \Pr{X=k_n})}.
	\end{align*}
	 Set $x := n^{-\alpha} \ell$. It follows that if $k_n \to \infty$ then we have uniformly for all $ 0 \le k \le o(n^{\alpha})$
	\begin{align}
		\label{eq:eindh1}
	\Pr{\#_{k_n} P_n = k} = R + (1+o(1)) \sum_{\ell \in n^\alpha \Omega} \Pr{N_n = \ell} \frac{(x n^{\alpha} \Pr{X=k_n})^k}{k! \exp(x n^{\alpha} \Pr{X=k_n})} \frac{\Pr{\bar{S}_{\ell-k} = n - k k_n}}{\Pr{S_\ell = n}}.
	\end{align}
	By Gnedenko's local limit theorem~\cite[Thm. 4.2.1]{MR0322926}  we also know that \begin{align*}
		\Pr{S_\ell = n} &= \ell^{-1/\alpha}(f(n / \ell^{1/\alpha}) + o(1)) = \ell^{-1/\alpha}f(n / \ell^{1/\alpha})(1+ o(1))
	\end{align*} uniformly for $\ell \in n^{\alpha} \Omega$. By a straight-forward  adaption of the proof of~\cite[Thm. 4.2.1]{MR0322926} it follows that the same local limit theorem also holds for the sums of $\bar{X}_1, \bar{X}_2, \ldots$. 
Hence, under the assumption $k_n \to \infty$ it follows that uniformly for all $0 \le k \le o( \min(n^\alpha, n/k_n) ) $ 
\begin{align}
					\label{eq:eindh2}
\frac{\Pr{\bar{S}_{\ell-k} = n - k k_n}}{\Pr{S_\ell = n}} = 1+ o(1).
\end{align}

Now, suppose that $k_n$ is an arbitrary sequence satisfying $n^{\alpha} \Pr{X=k_n} \to \upsilon \in ]0, \infty[$. Then $k_n = \Theta(n^{\alpha / (1+\alpha)})$. Combining Equations~\eqref{eq:eindh1} and~\eqref{eq:eindh2} and the local limit theorem for $N_n$ from Theorem~\ref{te:dilute} it follows that at least for $k$ constant 
\begin{align*}
	\Pr{\#_{k_n} P_n = k} &= R + (1+o(1)) \sum_{\ell \in n^\alpha \Omega} \Pr{N_n = \ell} \frac{(x \upsilon)^k}{k! \exp(x \upsilon)} \\
	&= R +   \Exb{ \frac{(\upsilon Z)^k}{k! \exp(\upsilon Z)}, Z \in \Omega} + o(1).
\end{align*}
Since we may take $\epsilon>0$ arbitrarily small (and chose $\Omega$ accordingly), it follows that 
\begin{align}
	\Pr{\#_{k_n} P_n = k}&=   \Exb{ \frac{(\upsilon Z)^k}{k! \exp(\upsilon Z)}} + o(1).
\end{align}

Next, suppose that $k_n$ satisfies $n^{\alpha} \Pr{X=k_n} \to 0$ instead. If $k_n \to \infty$, then by the same arguments it follows that
\begin{align}
	\label{eq:zeronaught}
\Pr{\#_{k_n} P_n = 0} = 1 + o(1).
\end{align}
If  $k_n$ does not tend to infinity, then the assumption $n^{\alpha} \Pr{X=k_n} \to 0$ implies that for each large enough $n$ either $k_n$ belongs to some of the finitely many values $r$ with $\Pr{X= r} = 0$, or $n$ belongs to a fixed subsequence along which $k_n$ tends to infinity. Hence~\eqref{eq:zeronaught} still holds.

We have thus completed the verification of Equation~\eqref{eq:refer2}.  It remains to check Equation~\eqref{eq:refer1}. To this end, suppose that $k_n$ is an arbitrary sequence satisfying $n^{\alpha} \Pr{X=k_n} \to \infty$. We know that $N_n / n^\alpha \convdis Z$ by Theorem~\ref{te:dilute}. By Slutsky's theorem it follows that in order to show~\eqref{eq:refer1} it suffices to verify that
\begin{align}
	\label{eq:priel}
	\frac{\#_{k_n} P_n}{\Pr{X=k_n} N_n} \convp 1.
\end{align}

In order to do so, we first consider the subcase for which $k_n$ is constant. For ease of notation, we set 
\[
\mathfrak{A}_\ell = \{ k \in \ndN_0 \mid |k/(\ell \Pr{X=k_n}) - 1| > \delta \}.
\] For any $\delta>0$ we have
\begin{align}
	\label{eq:arganalougous}
	\Prb{\left|\frac{\#_{k_n} P_n}{\Pr{X=k_n} N_n} -1 \right| > \delta} &= \Prb{\left|\frac{\#_{k_n} P_n}{\Pr{X=k_n} N_n} -1 \right| > \delta, N_n \notin  n^{\alpha} \Omega } \\&\quad+ \sum_{\ell \in n^\alpha \Omega} \Pr{N_n = \ell} \frac{\Prb{S_\ell = n, |\{ 1 \le i \le \ell \mid X_i = k_n\}| \in \mathfrak{A}_\ell }}{\Pr{S_\ell = n}}. \nonumber
\end{align}
The first summand tends to zero by Theorem~\ref{te:dilute}. The second summand tends to zero because the denominator satisfies $\Pr{S_\ell = n} = \Theta(n^{-1})$ by the local limit theorem and the numerator tends to zero way faster (comparably to  $\exp( - \Theta(n^\alpha) \Pr{X=k_n})$) by the Chernoff bounds. This verifies Equation~\eqref{eq:priel} for $k_n$ constant. (In fact, this argument would still work as long as $\frac{n^{\alpha} \Pr{X=k_n}}{\log n} \to \infty$, but we are not going to use this here.)

In order to verify Equation~\eqref{eq:priel} in the remaining cases we may without loss of generality assume that $k_n \to \infty$. Indeed, if Equation~\eqref{eq:priel} holds for constants, then there is some sequence $k_n' \to \infty$ so that for each $\delta>0$
\begin{align}
	\label{eq:substack}
	\sum_{\substack{0 \le k' \le k_n' \\ \Pr{X=k'}>0}} \Prb{\left|\frac{\#_{k'} P_n}{\Pr{X=k'} N_n} -1 \right| > \delta} \to 0.
\end{align}
We may hence split an arbitrary sequence $k_n$ into two subsequences, one where $k_n \le k_n'$, and one where $k_n > k_n'$.  Thus, it really suffices to show ~\eqref{eq:priel} for  the case $k_n \to \infty$. 

Let us hence consider the remaining case where $k_n \to \infty$. It follows easily from~\eqref{eq:substack} that with high probability $\#_{k_n} P_n < N_n / 2$. Set $\mathfrak{A}_\ell' = \{k \in \mathfrak{A}_\ell \mid k < \ell/2\}$. Analogously as in Equation~\eqref{eq:analogouspriel} we may write
\begin{align*}
	\Prb{\left|\frac{\#_{k_n} P_n}{\Pr{X=k_n} N_n} -1 \right| > \delta} &= o(1) + \sum_{\substack{\ell \in n^\alpha \Omega \\ k \in \mathfrak{A}_\ell'}} \Pr{N_n = \ell}  \binom{\ell}{k} \Pr{X=k_n}^k (1- \Pr{X=k_n})^{\ell-k} \frac{\Pr{\bar{S}_{\ell-k} = n - k k_n}}{\Pr{S_\ell = n}}.
\end{align*}
Here the $o(1)$ term corresponds to the events that either $N_n \notin n^\alpha \Omega$ or $\#_{k_n} P_n \ge N_n / 2$, both of which have probabilities that tend to zero. Since $k_n \to \infty$ and since we only consider $k< \ell/2$ it follows analogously as for Equation~\eqref{eq:eindh2} (that is, by the local limit theorems for the nominator and the denominator) that
\begin{align}
	\label{eq:eindh3}
	\frac{\Pr{\bar{S}_{\ell-k} = n - k k_n}}{\Pr{S_\ell = n}} = O(1)
\end{align}
uniformly for all $\ell \in n^\alpha \Omega$ and $k \in \mathfrak{A}_\ell'$. Hence, using $n^{\alpha} \Pr{X=k_n} \to \infty$ and the Chernoff bounds, 
\begin{align*}
	\Prb{\left|\frac{\#_{k_n} P_n}{\Pr{X=k_n} N_n} -1 \right| > \delta} &= o(1) + O(1) \sum_{\ell \in n^\alpha \Omega} \Pr{N_n = \ell} \Pr{\mathrm{Bin}(\ell, \Pr{X=k_n}) \in \mathfrak{A}_\ell'} \\
	&= o(1),
\end{align*}
with $\mathrm{Bin}(\ell, \Pr{X=k_n})$ denoting some binomial random variable with the corresponding parameters. This completes the verification of~\eqref{eq:priel} and hence concludes the proof of Equation~\eqref{eq:refer1}.
\end{proof}

\begin{proof}[Proof of Corollary~\ref{co:pointprocess}]
	Let $\psi \ge 0$ denote a non-negative Riemann integrable function whose support is a compact subset of the interval $]0,1]$. Set
	\[
		P_n^\psi := \sum_{i=1}^{N_n} \psi(K_i / n).
	\]
	By Theorem~\ref{te:dilute} there exists a sufficiently large compact interval $\Omega \subset ]0, \infty[$ such that \[
	\Pr{N_n/n^{\alpha} \in \Omega} > 1 - \epsilon
	\] for all sufficiently large $n$. As before, we let $n^\alpha \Omega$ denote the collection of all integers $\ell$ with $\ell / n^\alpha \in \Omega$. Then
	\begin{align}
		\label{eq:lasik1}
		\Exb{P_n^\psi} &= R + \sum_{\ell \in n^\alpha \Omega} \Pr{N_n = \ell} \ell \Exb{\psi(K_1/n) \mid N_n = \ell}
	\end{align}
	with
	\[
		R := \Exb{P_n^\psi, N_n \notin n^\alpha \Omega}.
	\]
	Since we assumed $\psi$ to have  compact support contained in $]0,1]$, it follows that there exists  $0<\delta<1$ with $\psi(y)=0$ whenever $y<\delta$. Since $\sum_{i=1}^{N_n} K_i = n$, it follows that $K_i \ge \delta n$ only for at most $1/\delta$ many $i$. Hence
	\begin{align}
		\label{eq:duetovaliente}
		P_n^\psi \le \delta^{-1} \sup \psi,
	\end{align}
and	$\sup \psi < \infty$ since $\psi$ is Riemann integrable. Thus
	\begin{align}
		\label{eq:rakaka}
		R < \epsilon  \delta^{-1} \sup \psi.
	\end{align}
	Recall that $S_\ell$ denotes the sum of $\ell$ independent copies of $X$. 	Using Gnedenko's local limit theorem~\cite[Thm. 4.2.1]{MR0322926} and setting $x := \ell / n^{\alpha}$ we have uniformly for all $\ell \in n^\alpha \Omega$
	\begin{align*}
		\Exb{\psi(K_1/n) \mid N_n = \ell} &= \sum_{k =0}^n \psi(k/n) \Pr{X=k}  \frac{\Pr{S_{\ell -1}= n-k}}{\Pr{S_\ell= n}} \\
		&= \sum_{k=0}^n \psi(k/n) \Pr{X=k}  \frac{(\ell-1)^{-1/\alpha}(f((n-k)    / (\ell-1)^{1/\alpha}) + o(1))}{\ell^{-1/\alpha}(f(n / \ell^{1/\alpha}) + o(1))} \\
		&= \frac{1}{f(x^{-1/\alpha})}\sum_{k = 0}^n \psi(k/n) \Pr{X=k} (f(x^{-1/\alpha}(1 - k/n)) + o(1)).
	\end{align*}
	We set $y := k /n$.  For $y< \delta$ we have $\psi(k/n)=0$, and for $y\ge \delta$ we  have  that $k$ becomes large as $n \to \infty$ and hence
	\[
		\Pr{X=k} \sim \frac{c_w}{W(\rho_w)} k^{-\alpha -1} \sim \frac{c_w}{W(\rho_w)} y^{-\alpha -1} n^{-\alpha -1}.
	\]
	Hence
	\begin{align}
	\Exb{\psi(K_1/n) \mid N_n = \ell} =   n^{-\alpha -1} \frac{c_w}{W(\rho_w)}  \frac{1}{f(x^{-1/\alpha})}\sum_{k = \lfloor \delta n \rfloor}^n \psi(y) y^{-\alpha-1} (f(x^{-1/\alpha}(1 - y)) + o(1)),
	\end{align}
with an $o(1)$ term that is uniform in $n$. Plugging this into Equation~\eqref{eq:lasik1} and applying Theorem~\ref{te:dilute} yields
\begin{align}
	\Exb{P_n^\psi} &= R + n^{-\alpha-1} \sum_{\ell \in n^\alpha \Omega}  \frac{c_w}{W(\rho_w)}  \frac{x \tilde{f}(x)}{f(x^{-1/\alpha})}\sum_{k = \lfloor \delta n \rfloor}^n \psi(y) y^{-\alpha-1} (f(x^{-1/\alpha}(1 - y)) + o(1)). \\
	&= R + o(1) + \frac{c_w}{W(\rho_w)} \frac{1}{\alpha \Ex{(X_{\alpha}(\gamma, 1,0))^{\alpha(b-1)}}}  \int_\Omega \int_{[\delta,1]} \frac{1}{x^{b+1/\alpha-1}} \frac{\psi(y)}{y^{\alpha+1}} f(x^{-1/\alpha}(1 - y)) \,\mathrm{d}y\,\mathrm{d}x. \nonumber
\end{align}
We may write $\Omega =[\delta_1,\delta_2]$ for some $0< \delta_1<\delta_2$. By Fubini's theorem and  integration by the substitution $t=(1-y)/x^{1/\alpha}$, it follows that
\[
\int_\Omega \int_{[\delta,1]} \frac{1}{x^{b+1/\alpha-1}} \frac{\psi(y)}{y^{\alpha+1}} f(x^{-1/\alpha}(1 - y)) \,\mathrm{d}y\,\mathrm{d}x = \int_\delta^1 \frac{\psi(y)}{y^{\alpha+1}} \int_{(1-y)/\delta_2^{1/\alpha}}^{(1-y)/\delta_1^{1/\alpha}} \alpha f(t) \frac{ (1-y)^{\alpha(2-b-1/\alpha)}}{t^{\alpha(2-b)}} \,\mathrm{d}t\,\mathrm{d}y.
\]
Since $\epsilon>0$ was arbitrary, and since we may always replace $\Omega$ by a larger compact interval, and since $\alpha(2-b) \in ]0,1[$ ensures that $\int_{0}^\infty f(t) t^{-\alpha(2-b)}\,\mathrm{d}t < \infty$ (see ~\cite[Thm. 5.1, Ex. 5.5]{2011arXiv1112.0220J}),  it follows that 
\begin{align}
	\label{eq:intensityfunc}
	\lim_{n \to \infty} \Exb{P_n^\psi} = \frac{c_w}{W(\rho_w)} \frac{\Ex{(X_{\alpha}(\gamma, 1,0))^{-\alpha(2-b)}}}{ \Ex{(X_{\alpha}(\gamma, 1,0))^{\alpha(b-1)}}}  \int_0^1 \frac{\psi(y)}{y^{\alpha+1} (1-y)^{1-\alpha(2-b)} } \,\mathrm{d}y.
\end{align}
By Equations~\eqref{eq:lambdadilute} and \eqref{eq:xalphadilute}, we get 
\begin{align*}
\frac{c_w}{W(\rho_w)} \frac{\Ex{(X_{\alpha}(\gamma, 1,0))^{-\alpha(2-b)}}}{ \Ex{(X_{\alpha}(\gamma, 1,0))^{\alpha(b-1)}}} &= \frac{c_w}{W(\rho_w)} \lambda^{-1} \frac{\Gamma(3-b) \Gamma(1-\alpha(b-1))}{\Gamma(1+\alpha(2-b)) \Gamma(2-b)} = \frac{\Gamma(1- \alpha(b-1)) }{\Gamma(1-\alpha) \Gamma(\alpha(2-b))} 
\end{align*}
and hence
\begin{align}
	\label{eq:intensityfunc2}
	\lim_{n \to \infty} \Exb{P_n^\psi} = \frac{1}{B(1-\alpha, \alpha(2-b))}  \int_0^1 \frac{\psi(y)}{y^{\alpha+1} (1-y)^{1-\alpha(2-b)} } \,\mathrm{d}y.
\end{align}

Let us now write $P_n^{I} = P_n^\psi$ for the special case where $\psi = \one_{I}$ is the indicator function for a compact interval $I$ satisfying $I \subset ]0,1]$. We continue using the same notation, in particular $0<\delta<1$ satisfies $I \cap [0,\delta[ = \emptyset$. For integers $m \ge 1$ we set $(P_n^{I})_{m} = P_n^I (P_n^I -1) \cdots (P_n^I - m +1)$. Then arguing analogously as before we obtain
\begin{align*}
		\Exb{(P_n^I)_m} &= R_m + \sum_{\ell \in n^\alpha \Omega} \Pr{N_n = \ell} \ell ( \ell-1) \cdots (\ell -m +1) \Exb{ \one_I(K_1/n) \cdots \one_I(K_m/n) \mid N_n = \ell} \\
						&= R_m + (1+o(1))\sum_{\ell \in n^\alpha \Omega} \tilde{f}(x) x^m n^{(m-1)\alpha} \Exb{ \one_I(K_1/n) \cdots \one_I(K_m/n) \mid N_n = \ell}
\end{align*}
with
\[
	R_m := \Exb{(P_n^I)_m, N_n \notin n^\alpha \Omega} \le \delta^{-m} \epsilon.
\]
Setting $y_i = k_i / n$ for $1 \le i \le m$, and $nI = \{k \in \ndZ \mid k/n \in I\}$, we obtain analogously as before
\begin{align*}
	&\Exb{ \one_I(K_1/n) \cdots \one_I(K_m/n) \mid N_n = \ell} \\
	&\qquad= \sum_{k_1, \ldots, k_m \in nI} \Pr{X=k_1} \cdots \Pr{X=k_m}  \frac{\Pr{S_{\ell -m}= n-k_1 - \ldots - k_m}}{\Pr{S_\ell= n}} \\
	&\qquad= \left( \frac{c_w}{W(\rho_w)}  \right)^m n^{-m(1+\alpha)} \sum_{k_1, \ldots, k_m \in nI} y_1^{-1-\alpha} \cdots y_m^{-1-\alpha} \frac{f(x^{-1/\alpha}(1-y_1 - \ldots - y_m)) + o(1)}{f(x^{-1/\alpha})}.
\end{align*}
Hence
\begin{align*}
	&\Exb{(P_n^I)_m} \\
	&= R_m + o(1) + \left( \frac{c_w}{W(\rho_w)}  \right)^m \int_I \cdots \int_I \int_\Omega \tilde{f}(x) x^m y_1^{-1-\alpha} \cdots y_m^{-1-\alpha} \frac{f(x^{-1/\alpha}(1-y_1 - \ldots - y_m))}{f(x^{-1/\alpha})} \,\mathrm{d}x\,\mathrm{d}y_1 \cdots \mathrm{d}y_m \\
		&= R_m + o(1) + \frac{\left( \frac{c_w}{W(\rho_w)}  \right)^m }{\alpha  \Ex{(X_{\alpha}(\gamma, 1,0))^{\alpha(b-1)}}} \int_I \cdots \int_I \int_\Omega  y_1^{-1-\alpha} \cdots y_m^{-1-\alpha} \frac{f(x^{-1/\alpha}(1-y_1 - \ldots - y_m))}{x^{b+1/\alpha-m}} \,\mathrm{d}x\,\mathrm{d}y_1 \cdots \mathrm{d}y_m.
\end{align*}
Recall that $f(t)=0$ for $t \le 0$. Integrating by the substitution $t = x^{-1/\alpha}(1- y_1 - \ldots - y_m)$ yields for $y_1 + \ldots + y_m <1$
\begin{align*}
\int_\Omega   \frac{f(x^{-1/\alpha}(1-y_1 - \ldots - y_m))}{x^{b+1/\alpha-m}} \,\mathrm{d}x &= \alpha  \int_{(1-y_1 - \ldots - y_m)/\delta_2^{1/\alpha}}^{(1-y_1 - \ldots - y_m)/\delta_1^{1/\alpha}} f(t) \frac{ (1-y_1 - \ldots - y_m)^{\alpha(m+1)- \alpha b-1}}{t^{\alpha(m+1-b)}} \,\mathrm{d}t.
\end{align*}
Hence, arguing as before, we obtain
\begin{multline}
	\label{eq:facmom}
	\lim_{n \to \infty} \Exb{(P_n^I)_m } = \\ \left(\frac{c_w}{W(\rho_w)}\right)^m \frac{\Ex{(X_{\alpha}(\gamma, 1,0))^{-\alpha(m+1-b)}}}{ \Ex{(X_{\alpha}(\gamma, 1,0))^{\alpha(b-1)}}}  \int_I \cdots \int_I \one_{y_1 + \ldots + y_m \le 1}\frac{(1-y_1 - \ldots - y_m)^{\alpha(m+1-b)-1}}{y_1^{\alpha+1}\cdots y_m^{\alpha+1}  } \,\mathrm{d}y_1 \cdots \mathrm{d}y_m.
\end{multline}
By Equation~\eqref{eq:duetovaliente} it follows that $(P_n^I)_m < \delta^{-m}$. Hence the method of moments  applies, yielding the existence of a random variable $P_\infty^I$ with
\begin{align}
	\label{eq:pniconv}
	P_n^I \convdis P_\infty^I \qquad \text{and} \qquad \lim_{n \to \infty} \Exb{(P_n^I)_m } = \Ex{(P_\infty^I)_m}, \quad m \ge 1
\end{align}
and
\begin{multline}
	\label{eq:wellitried}
	\Ex{(P_\infty^I)_m} = \\ \left(\frac{c_w}{W(\rho_w)}\right)^m \frac{\Ex{(X_{\alpha}(\gamma, 1,0))^{-\alpha(m+1-b)}}}{ \Ex{(X_{\alpha}(\gamma, 1,0))^{\alpha(b-1)}}}  \int_I \cdots \int_I \one_{y_1 + \ldots + y_m \le 1} \frac{(1-y_1 - \ldots - y_m)^{\alpha(m+1-b)-1}}{y_1^{\alpha+1}\cdots y_m^{\alpha+1}  } \,\mathrm{d}y_1 \cdots \mathrm{d}y_m.
\end{multline}
for all integers $m \ge 1$. By the same arguments, it follows that~\eqref{eq:pniconv} still holds when $I$ is a finite union of intervals that are subsets of $]0,1]$ so that $I$ is disjoint to some neighbourhood of $0$. This in turn implies that for each step function $\psi: ]0,1] \to \ndR_{\ge 0}$ with compact support the method of moments applies to $P_n^\psi$, yielding that there is a limiting random variable $P_\infty^\psi$ such that 
\begin{align}
	\label{eq:wwwwinvint}
	P_n^\psi \convdis P_\infty^\psi \qquad \text{and} \qquad \lim_{n \to \infty} \Exb{(P_n^\psi)_m } = \Ex{(P_\infty^\psi)_m}, \quad m \ge 1.
\end{align}
Now, suppose that $\psi: ]0,1] \to \ndR_{\ge 0}$ is continuous with compact support. By standard arguments involving approximations by step functions, it follows that~\eqref{eq:wwwwinvint} still holds for $\psi$ as well. Indeed, we may select $\delta>0$ with $\psi(x) =0$ for all $x \le \delta$. We may furthermore select an increasing sequence of step functions $\psi_1 \le \psi_2 \le \ldots \le \psi$ with $\psi_k(x) = 0$ for all $x \le \delta$ and $k \ge 1$, and $\sup | \psi_k - \psi| \to 0$ as $k \to \infty$. This way, $|P_n^\psi - P_n^{\psi_k}| \le \delta^{-1} \sup | \psi - \psi|$. Since both $P_n^\psi$ and $P_n^{\psi_k}$ are bounded by $ \delta^{-1} \sup \psi$, it follows that for any integer $\ell \ge 1$ it holds that \[\Ex{(P_n^{\psi_k})^\ell} = \Ex{(P_n^{\psi})^\ell} + O(\sup | \psi_k - \psi|),\] with the $O$-term being uniform in $k$ and $n$. Since~\eqref{eq:wwwwinvint} holds for step functions, it follows that \[
\Ex{(P_\infty^{\psi_k})^\ell} = \limsup_{n \to \infty} \Ex{(P_n^{\psi})^\ell} + O(\sup | \psi_k - \psi|).\]
The term on the left side of this equation is increasing in $k$ and bounded by $(\delta^{-1} \sup \psi)^\ell$. Taking $k \to \infty$ on both sides, and repeating the same argument with $\liminf_{n \to \infty}$, it follows that
\[
 \lim_{k \to \infty} \Ex{(P_\infty^{\psi_k})^\ell} =  \lim_{n \to \infty} \Ex{(P_n^{\psi})^\ell}.
\]
Thus, similarly as before, the method of moments applies and~\eqref{eq:wwwwinvint} holds.

Having verified~\eqref{eq:wwwwinvint}  for any continuous function $\psi: ]0,1] \to \ndR_{\ge 0}$  with compact support, we may apply~\cite[Cor. 4.14]{zbMATH06684669}, readily yielding the existence of a point process $\Upsilon$ with $\Upsilon_n \convdis \Upsilon$, and by~\eqref{eq:intensityfunc} its intensity is given by
	\[
     \frac{1}{B(1-\alpha, \alpha(2-b)) x^{\alpha+1} (1-x)^{1-\alpha(2-b)} }\,\mathrm{d}x.
\]

It remains to verify that almost surely $\Upsilon(]0,1]) = \infty$. To this end, let $\eta_1 \ge \eta_2 \ge \ldots$ denote the ranked points of $\Upsilon$, with $\eta_k := 0$ whenever $\Upsilon(]0,1]) <k$. Suppose that there exists an integer $k \ge 1$ with 
\[
\delta_0 := \Pr{\eta_k = 0}>0.
\]
By similar arguments as in~\cite[Lem. 4.4]{zbMATH01877115} it  follows from $\Upsilon_n \convdis \Upsilon$ that the $k$th largest component $K_{(k)}$ of $P_n$ satisfies
\[
	K_{(k)} /n \convdis \eta_k.
\]
Hence for any $\epsilon>0$ we have 
\[
	\liminf_{n \to \infty} \Pr{K_{(k)} /n  < \epsilon}  \ge \delta_0.
\]
Thus, there exists a sequence $a_n>0$ with $a_n \to 0$ and 
\begin{align}
	\label{eq:alealealign}
	\Pr{K_{(k)} /n  < a_n} \ge \delta_0/2
\end{align}
for all $n \ge 1$. We may replace $a_n$ by any sequence $a_n' \ge a_n$ with $a_n' = o(1)$. Hence, without loss of generality we may assume that 
\[
	a_n^{1+\alpha} n \to \infty.
\]
Let $C>0$ denote an arbitrary constant, and set 
\[
	b_n = (1/a_n^\alpha - C)^{-1/\alpha}
\]
so that $a_n \sim b_n$, and for all large enough integers $n$  
\[
	a_n < b_n \qquad \text{and} \qquad \frac{1}{a_n^\alpha} - \frac{1}{b_n^\alpha} = C.
\]
Set ${I_n} = [a_n, b_n]$. Arguing as before (and using the same notation) we arrive at
\begin{align*}
	\Exb{(P_n^{I_n})_m} &=  R_m + (1+o(1))\sum_{\ell \in n^\alpha \Omega} (\tilde{f}(x) +o(1)) x^m n^{(m-1)\alpha} \Exb{ \one_{I_n}(K_1/n) \cdots \one_{I_n}(K_m/n) \mid N_n = \ell}
\end{align*}
with $R_m  = \Exb{(P_n^{I_n})_m, N_n \notin n^\alpha \Omega}$
and
\begin{multline*}
	 \Exb{ \one_{I_n}(K_1/n) \cdots \one_{I_n}(K_m/n) \mid N_n = \ell} \\ \sim \left( \frac{c_w}{W(\rho_w)}  \right)^m n^{-m(1+\alpha)} \sum_{k_1, \ldots, k_m \in n{I_n}} y_1^{-1-\alpha} \cdots y_m^{-1-\alpha} \frac{f(x^{-1/\alpha}(1-y_1 - \ldots - y_m)) + o(1)}{f(x^{-1/\alpha})}.
\end{multline*}
Since $n b_n = o(n)$ we have $y_i = o(1)$ uniformly for all $1 \le i \le m$  and $k_i \in n{I_n}$ in the sum.  It follows that uniformly for all sum indices $k_1, \ldots, k_n \in n{I_n}$
\begin{align}
	 \frac{f(x^{-1/\alpha}(1-y_1 - \ldots - y_m)) + o(1)}{f(x^{-1/\alpha})} = 1 + o(1).
\end{align}
By the choice of $a_n$ and $b_n$, and using $n a_n^{1+\alpha} \to \infty$,  it follows that
\[
	\frac{1}{n} \sum_{k \in n I_n} (k/n)^{-1-\alpha} = o(1) + \int_{I_n} y^{-1-\alpha}\,\mathrm{d}y   = o(1) +  C/\alpha.
\]
Hence
\begin{align*}
	\Exb{ \one_{I_n}(K_1/n) \cdots \one_{I_n}(K_m/n) \mid N_n = \ell} &\sim \left( \frac{c_w}{W(\rho_w)}  \right)^m n^{-m\alpha} \left(\int_{I_n} y^{-1-\alpha}\,\mathrm{d}y \right)^m \\
	&\sim \left( \frac{c_w}{W(\rho_w)}  \right)^m n^{-m\alpha} \left(\frac{C}{\alpha}\right)^m
\end{align*}
Hence we arrive at
\begin{align}
	\label{eq:nononononoyes}
	\Exb{(P_n^{I_n})_m} &=  R_m + o(1) + \sum_{\ell \in n^\alpha \Omega} \tilde{f}(x) x^m n^{-\alpha} \left( \frac{c_w}{W(\rho_w)}  \right)^m n^{-m\alpha} \left(\frac{C}{\alpha}\right)^m \\
	&= R_m + o(1) +  \Ex{Z^m, Z \in \Omega} \left( \frac{c_w}{W(\rho_w)}  \right)^m  \left(\frac{C}{\alpha}\right)^m. \nonumber
\end{align}
As for $R_m$,  we may write
\begin{align*}
	R_m &\le \sum_{\ell = 1}^n \one_{\ell \notin n^\alpha \Omega} \frac{\Pr{N=\ell} \Pr{S_\ell= n}}{\Pr{S_N = n}} \ell^m  \Exb{ \one_{I_n}(K_1/n) \cdots \one_{I_n}(K_m/n) \mid N_n = \ell}.
\end{align*}
Using that $f$ is continuous with $f(0)=0$ and $\lim_{t \to \infty}f(t) = \infty$, we obtain
\begin{align*}
	&\Exb{ \one_{I_n}(K_1/n) \cdots \one_{I_n}(K_m/n) \mid N_n = \ell} \\
	&\quad= \left( \frac{c_w}{W(\rho_w)}  \right)^m O(n^{-m(1+\alpha)}) \sum_{k_1, \ldots, k_m \in n{I_n}} y_1^{-1-\alpha} \cdots y_m^{-1-\alpha} \frac{\ell^{-1/\alpha} (f(x^{-1/\alpha}(1-y_1 - \ldots - y_m)) + o(1))}{\Pr{S_\ell= n}} \\
	&\quad= \left( \frac{c_w}{W(\rho_w)}  \right)^m O(n^{-m\alpha}) \ell^{-1/\alpha} \frac{f(x^{-1/\alpha})}{\Pr{S_\ell= n}} \left(\frac{C}{\alpha}\right)^m.
\end{align*}
Using Equation~\eqref{eq:stoppedsumdilute}, it follows that
\begin{align*}
	R_m &= O(n^{-\alpha}) \sum_{\ell = 1}^n \one_{\ell \notin n^\alpha \Omega}  \frac{L_v(\ell)}{L_v(n^\alpha)} x^m \tilde{f}(x)  \left( \frac{c_w}{W(\rho_w)}  \right)^m  \left(\frac{C}{\alpha}\right)^m.
\end{align*}
By the Potter bounds, it holds for any $\tilde{\epsilon}>0$ that
\[
	\frac{L_v(\ell)}{L_v(n^\alpha)} \le O(1) \max(x^{\tilde{\epsilon}}, x^{-\tilde{\epsilon}}).
\]
Thus
\begin{align}
	\label{eq:yeyeyeno}
	R_m = O(1) \max\left( \Exb{Z^{m+\tilde{\epsilon}}, Z \notin \Omega}, \Exb{Z^{m-\tilde{\epsilon}}, Z \notin \Omega} \right).
\end{align}
Since both  Equation~\eqref{eq:nononononoyes} and Equation~\eqref{eq:yeyeyeno} hold for all compact subsets $\Omega \subset ]0, \infty[$, it follows that
\begin{align}
	\lim_{n \to \infty}\Exb{(P_n^{I_n})_m} =  \Ex{Z^m} \left( \frac{c_w}{W(\rho_w)}  \right)^m  \left(\frac{C}{\alpha}\right)^m.
\end{align}
This is equal to the $m$th factorial moment of the random variable $\mathrm{Poi}(\upsilon Z)$ for $\upsilon= \frac{c_w C}{\alpha W(\rho_w)} $. By the method of moments, it follows that
\begin{align}
	\label{eq:intermediate}
	P_n^{I_n} \convdis \mathrm{Poi}(\upsilon Z).
\end{align}
Using Inequality~\eqref{eq:alealealign} it follows that
\begin{align*}
	\delta_0/ 2 &\le \Pr{K_{(k)} /n  < a_n} \\
	&= \Pr{P_n^{[a_n, 1]} < k} \\
	&\le \Pr{P_n^{I_n} < k}
\end{align*}
and hence
\[
\delta_0/ 2  \le \Pr{ \mathrm{Poi}(\upsilon Z) <k }.
\]
But this holds for any $C>0$ and thus for any $\upsilon>0$. We may choose $C$ (and hence $\upsilon$) large enough so that $\Pr{ \mathrm{Poi}(\upsilon Z) <k }<\delta_0/ 2$, yielding a contradiction. Hence the premise that $\Pr{\eta_k =0}>0$ was false, yielding $\eta_k >0$ almost surely. As this holds for any $k \ge 1$, we have $\Upsilon(]0,1]) = \infty$ almost surely. This completes the proof.
\end{proof}

\begin{proof}[Proof of Proposition~\ref{pro:explicit}]
Let $0<\delta\le 1$ be given, and set $I=[\delta,1]$. We will continue using the notation of the proof of Corollary~\ref{co:pointprocess}. From Equation~\eqref{eq:facmom} and its two preceding equations it follows that	
\begin{align}
\Ex{(P_\infty^I)_m} &= \frac{\left( \frac{c_w}{W(\rho_w)}  \right)^m }{\alpha  \Ex{(X_{\alpha}(\gamma, 1,0))^{\alpha(b-1)}}} \int_I \cdots \int_I \int_0^\infty  \frac{f(x^{-1/\alpha}(1-y_1 - \ldots - y_m))}{(y_1^{1+\alpha} \cdots y_m^{1+\alpha}) x^{b+1/\alpha-m}} \,\mathrm{d}x\,\mathrm{d}y_1 \cdots \mathrm{d}y_m \\
&= \frac{\left( \frac{c_w}{W(\rho_w)}  \right)^m }{  \Ex{(X_{\alpha}(\gamma, 1,0))^{\alpha(b-1)}}} \int_0^\infty\int_I \cdots \int_I   \frac{f(u(1-y_1 - \ldots - y_m))}{(y_1^{1+\alpha} \cdots y_m^{1+\alpha}) u^{\alpha(m+1-b)}} \mathrm{d}y_1 \cdots \mathrm{d}y_m\,\mathrm{d}u. \nonumber
\end{align}
We know that this integral converges since $P_\infty^I \le 1/\delta$. Alternatively, we could argue that it converges due to the asymptotics of $f(x)$ near $x=0$ and $x=\infty$: By~\cite[Thm. 4.10, Eq. (3.27), (3.57)]{2011arXiv1112.0220J} it holds that $f(x) = \frac{1}{\pi} \Gamma(1+\alpha) \sin(\pi \alpha ) x^{-1-\alpha} + O(x^{-1-2\alpha})$ as $x \to \infty$. Moreover $f$ is infinitely differentiable and $f(x) = 0$ for $x \ge 0$ we also have for all $k \ge 1$ that $f(x) = O(x^k)$ as $x \to 0$. 

By~\cite[Ex. 3.16]{2011arXiv1112.0220J}, the Laplace transform of $X_{\alpha}(\gamma, 1,0)$ is given by
\[
	\Ex{ e^{-t X_{\alpha}(\gamma, 1,0)}} = \exp(-\lambda t^\alpha), \qquad \Re(t) \ge 0
\]
with $\lambda =  - \frac{c_w}{W(\rho_w)} \Gamma(-\alpha)>0$ as in Equation~\eqref{eq:lambdadilute}. Hence the characteristic function
\[
	\Ex{e^{isX_{\alpha}(\gamma, 1,0)}} = \exp\left(-\lambda |s|^\alpha\left( \cos\left(\frac{\alpha \pi}{2}\right) -  i \mathrm{sgn}(s) \sin\left( \frac{\alpha \pi}{2}  \right) \right)   \right), \qquad s \in \ndR
\]
is integrable. By the inversion formula,
\[
	f(x) = \frac{1}{2\pi} \int_{-\infty}^\infty e^{-ixt + \frac{c_w}{W(\rho_w)} \Gamma(-\alpha) (-it)^\alpha} \,\mathrm{d}t.
\]
Thus
\[
\Ex{(P_\infty^I)_m} = \frac{\left( \frac{c_w}{W(\rho_w)}  \right)^m }{ 2\pi \Ex{(X_{\alpha}(\gamma, 1,0))^{\alpha(b-1)}}}  \int_0^\infty \int_{-\infty}^\infty  \frac{1}{ u^{\alpha(m+1-b)}} \left( \int_I \frac{e^{ityu}}{y^{1+\alpha}}\,\mathrm{d}y\right)^m e^{-iut + \frac{c_w}{W(\rho_w)} \Gamma(-\alpha) (-it)^\alpha} \,\mathrm{d}t\,\mathrm{d}u.
\]
By Equations~\eqref{eq:xalphadilute} and~\eqref{eq:lambdadilute},
\begin{align*}
	\Ex{X_{\alpha}(\gamma, 1,0)^{\alpha(b-1)}} =  \left( \frac{c_w}{W(\rho_w)}  \right)^{b-1}  \left( \frac{\Gamma(1-\alpha)}{\alpha} \right)^{b-1} \frac{\Gamma(2-b)}{\Gamma(1-{\alpha(b-1)})}.
\end{align*}
Using the substitutions $\tilde{u} = u \left(\frac{c_w}{W(\rho_w)}\right)^{1/\alpha} $ and $\tilde{t} = t \left(\frac{c_w}{W(\rho_w)}\right)^{-1/\alpha}$, and then renaming $\tilde{u}, \tilde{t}$ to $u,t$ again yields
\begin{align}
	\Ex{(P_\infty^I)_m} = \frac{ \Gamma(1- \alpha(b-1)) }{ 2\pi (-\Gamma(-\alpha))^{b-1} \Gamma(2-b)}  \int_0^\infty \int_{-\infty}^\infty  \frac{1}{ u^{\alpha(m+1-b)}} \left( \int_I \frac{e^{ityu}}{y^{1+\alpha}}\,\mathrm{d}y\right)^m e^{-iut +  \Gamma(-\alpha) (-it)^\alpha} \,\mathrm{d}t\,\mathrm{d}u.
\end{align}
Thus
\begin{align*}
	\Exb{z^{P_\infty^I}} &= \sum_{m=0}^{\infty} \frac{1}{m!} \Exb{(P_\infty^I)_m} (z-1)^m \\
	&=  \frac{ \Gamma(1- \alpha(b-1)) }{ 2\pi |\Gamma(-\alpha)|^{b-1} \Gamma(2-b)}  \int_0^\infty \int_{-\infty}^\infty   u^{\alpha(b-1)} \exp\left( \frac{z-1}{u^\alpha}\int_I \frac{e^{ityu}}{y^{1+\alpha}}\,\mathrm{d}y -iut -  |\Gamma(-\alpha)| (-it)^\alpha\right)  \,\mathrm{d}t\,\mathrm{d}u. \nonumber
\end{align*}
The reason why we may interchange summation and integration is a bit delicate, since the absolute value of the integrand grows like $u^{\alpha(b-1)}$ as $u \to \infty$, and $\alpha(b-1) \in ]-1,0[$. We may pull the sum inside of the first integral, because $\int_{-\infty}^\infty \left( \int_I \frac{e^{ityu}}{y^{1+\alpha}}\,\mathrm{d}y\right)^m e^{-iut +  \Gamma(-\alpha) (-it)^\alpha} \,\mathrm{d}t = 0$ for all $m > 1/\delta$. We may then pull the sum into the second integral, because $|\left( \int_I \frac{e^{ityu}}{y^{1+\alpha}}\,\mathrm{d}y\right)^m e^{-iut +  \Gamma(-\alpha) (-it)^\alpha}| \le \left( \int_I \frac{1}{y^{1+\alpha}}\,\mathrm{d}y\right)^m \exp(- |\Gamma(-\alpha)|\cos(\pi \alpha / 2) |t|^\alpha)$. For similar reasons, we may differentiate with respect to $z$ below the two integrals, yielding
\begin{align*}
\Pr{\eta_k < x} &= \sum_{j=0}^{k-1}\frac{1}{j!} \frac{\mathrm{d}^j}{\mathrm{d} z^j} \Ex{z^{P_\infty^I}} \bigg\vert_{z=0} \\
&= \frac{ \Gamma(1- \alpha(b-1)) }{ 2\pi |\Gamma(-\alpha)|^{b-1} \Gamma(2-b)}  \int_0^\infty \int_{-\infty}^\infty   u^{\alpha(b-1)} \exp\left( -\frac{1}{u^\alpha}\int_I \frac{e^{ityu}}{y^{1+\alpha}}\,\mathrm{d}y -iut -  |\Gamma(-\alpha)| (-it)^\alpha\right)  \\&\quad\,\, \sum_{j=0}^{k-1}\frac{1}{j!} \left( \frac{1}{u^\alpha}\int_I \frac{e^{ityu}}{y^{1+\alpha}} \right)^j  \,\mathrm{d}t\,\mathrm{d}u. \nonumber
\end{align*}
This completes the proof.
\end{proof}

\subsection{Product structures and extended composition schemes}

\begin{proof}[Proof of Lemma~\ref{te:extended}]
	It suffices to treat the case $\ell = 2$. The case $\ell > 2 $ then follows easily by induction. It is easy to see that
	\[
	(P_1, P_2) \eqdist ( (A_1, A_2) \mid A_1 + A_2 = n)
	\]
	for random variables $A_1, A_2$ with probability generating functions \[
	\Ex{z^{A_i}} = \frac{W_i(\rho_o z)}{W_i(\rho_o)}
	\]
	for $i=1,2$. In particular
	\[
	\Ex{z^{A_1 + A_2}} = \frac{O(z\rho_o)}{O(\rho_o)}.
	\]
	By \cite[Lem. 4.9]{MR3097424} the assumptions imply that $(o_n)_{n \ge 0}$ satisfies the subexponentiality condition and that the ratio $o_n / r_n$  converges to a positive constant. Consequently, the limits \[
	p_1 = \lim_{n \to \infty} \frac{w_n^{(1)} O(\rho_o)}{o_n W_1(\rho_o)} = \lim_{n \to\infty} \frac{\Pr{A_1 = n}}{\Pr{A_1 + A_2 = n}}
	\] and \[
	p_2 = \lim_{n \to \infty} \frac{w_n^{(2)} O(\rho_o)}{o_n W_2(\rho_o)} = \lim_{n \to\infty} \frac{\Pr{A_2 = n}}{\Pr{A_1 + A_2 = n}}
	\] exist. By \cite[Lem. 4.9]{MR3097424} it follows that $p_1 + p_2 = 1$.
	
	Since the density of $A_1 + A_2$ satisfies the subexponentiality condition it follows for each integer $k \ge 0$
	\begin{align*}
		\Pr{P_2 = k} &= \Pr{A_2 = k} \frac{\Pr{A_1=n-k}}{\Pr{A_1 + A_2 = n}} \\
		&\sim \Pr{A_2 = k} \frac{\Pr{A_1=n-k}}{\Pr{A_1 + A_2 = n-k}} \\
		&\to \Pr{A_2 = k} p_1
	\end{align*}
	as $n \to \infty$. Likewise
	\[
	\lim_{n \to \infty} \Pr{P_1 = k} = \Pr{A_1 = k} p_2.
	\]
	In other words, with limiting probability $p_1$ the component $P_1$ is macroscopic and $P_2$ converges to $A_2$, and likewise with limiting probability $p_2$ the component $P_2$ is macroscopic and $P_1$ converges in distribution to $A_1$. This verifies~\eqref{eq:tvapproxproduct} and hence completes the proof.
\end{proof}

\section*{Acknowledgement}

I warmly thank the referee for the thorough reading of the manuscript and for the helpful comments.




\end{document}